\documentclass{article}
\usepackage{indentfirst,latexsym,bm,amsmath,graphicx,amssymb}
\usepackage{color}
  
\usepackage{framed}  
\usepackage{paralist}  
\usepackage{multirow}  
\usepackage{verbatim}
\usepackage{array}
\usepackage{booktabs}  

\usepackage{mathrsfs} 

\usepackage{epstopdf}
 \usepackage{mathptmx}      
 \usepackage{indentfirst,latexsym,bm,amsmath,graphicx,amsthm,amssymb}
\usepackage{color}

\usepackage{paralist}
\usepackage{multirow}
\usepackage{verbatim}
\usepackage{array}
\usepackage{booktabs}

\usepackage{mathrsfs}

\usepackage{hyperref}

\usepackage{graphicx}
\usepackage{subcaption}
\usepackage{epstopdf}
\newtheorem{theorem}{Theorem}
\newtheorem{proposition}{Proposition}
\newtheorem{definition}{Definition}
\newtheorem{corollary}{Corollary}

\newcommand{\C}{\mathcal{C}}
\newcommand{\G}{\mathcal{G}}

\renewcommand{\P}{\mathcal{P}}
\renewcommand{\S}{\mathcal{S}}
\newcommand{\U}{\mathcal{U}}

\newcommand{\tc}{{\tilde{c}}}
\newcommand{\ts}{{\tilde{s}}}
\newcommand{\tx}{{\tilde{x}}}
\newcommand{\tu}{{\tilde{u}}}
\newcommand{\au}{{\overleftarrow{u}}}
\newcommand{\atu}{{\overleftarrow{\tilde{u}}}}
\newcommand{\aw}{{\overleftarrow{w}}}

\newcommand{\tz}{{\tilde{z}}}

\newcommand{\tW}{{\tilde{W}}}

\newcommand{\rlambda}{{\bar{\lambda}}}
\newcommand{\rmu}{{\bar{\mu}}}

\newcommand{\sfZ}{{\mathsf{Z}}}
\newcommand{\sfs}{{\mathsf{\mathbf{s}}}}

\begin{document}

\title{Directed FCFS Infinite Bipartite Matching}


\author{Gideon Weiss \thanks{Research supported in part by
Israel Science Foundation Grant 286/13. Part of this work was done while the author was visiting the Simons Institute for the Theory of Computing.
}    
}



\date{}

\maketitle

\begin{abstract}
We consider an infinite sequence consisting  of agents of several types and  goods of several types, with a bipartite compatibility graph between agent and good types.  Goods are matched with agents that appear earlier in the sequence using FCFS matching, if such are available, and  are lost otherwise.   This model may be used for two sided queueing applications such as ride sharing, web purchases, organ transplants, and for  abandon at completion redundancy  parallel service queues.  For this model we calculate matching rates and delays.  These can be used to obtain waiting times and help with design questions for related service systems.
We explore some relations of this model to other FCFS stochastic matching models.

\end{abstract}

\begin{quotation}
\noindent
\begin{compactitem}[ ]
\item
  \ref{sec.introduction}. Introduction
  \item
  \ref{sec.previous}, The Model and Previous Results
  \item
  \ref{sec.detailed}. Detailed Markov Chains with Multi-Bernoulli Stationary Distributions
   \item
  \ref{sec.constant}. A Concise Markov Chain, the Normalizing Constant, and a Sanity Check
   \item
  \ref{sec.matchingrates}. Calculation of the Matching Rates
   \item
  \ref{sec.delays}. The Distribution and the Moments of  Delays
   \item
  \ref{sec.limits}. Light and Heavy Traffic, Relation to Symmetric FCFS Matching
   \item
  \ref{sec.example}. A Simple Illustrative Example
   \end{compactitem}
 \end{quotation}

\section{Introduction}
\label{sec.introduction}

Classical queueing models have random streams of customer arrivals and  fixed sets of servers.  Recently however,  more and more applications are better modeled by  two sided queues, where the fixed sets of servers are replaced by  streams of arriving servers.  The simplest example of such a system is the taxi stand, where cabs and passengers arrive at random times, and depart at the moment that they meet.  More general systems where two sided models seem appropriate include ride sharing systems, organ transplants,  web postings and searches, job markets, trading and auctions.   These models share a more symmetric role of the so called customers and servers, and in particular, in these models service durations are replaced by inter-arrival times.  In the simplest taxi stand model, the symmetric version where both cabs and passengers can wait, is always either transient or null-recurrent, and   to be useful the model needs to include finite waiting rooms or abandonments. In particular one may assume that cabs always wait, while passengers abandon and are lost if they do not find a cab at the stand.   This seems also to be  the model appropriate for ride sharing systems like Uber and Lyft,  for Transplants of perishable organs or blood transfusions, and for internet purchasing.   Our results in this paper are relevant for the following two sided parallel service model:  Agents of several types arrive to the system and wait for goods, and goods of several types arrive and match and depart with a compatible agent, or, if no match is immediately available for an arriving good, it is lost.  Compatibility is determined by a bipartite graph.  The system is controlled by a matching policy -- which of the available compatible agents is assigned the arriving good.  In this paper we focus on the first come first served (FCFS) policy:  The arriving good is assigned to the longest waiting compatible agent.  We do not however study this system directly, instead we consider the following discrete time matching model.

We consider a matching system in which agents of several types, indexed by $c_i\in \C =\{c_1,\ldots,c_I\}$  are matched with several types of goods, indexed by $s_j\in \S=\{s_1,\ldots,s_J\}$,   subject to a bipartite compatibility graph  $\mathcal{G}$, such that $(s_j,c_i)\in\mathcal{G}$ if agent type $c_i$ is compatible  with good type $s_j$.  We assume an infinite sequence of agents and goods, which is chosen i.i.d. from $\C \cup \S$.
Matching is done on a first come first served basis (FCFS), where goods are only matched to agents that appear earlier in the sequence, and goods that find no compatible match are lost. 
We call this a directed matching model, to distinguish it from a model where goods are not lost, and can be matched to agents that appear later in the sequence.
 This model is an exact discrete time analog of the two sided parallel service system if arrivals of agents and goods form independent Poisson processes.  Moreover, this discrete model has the advantage that it  captures the essentials of more general systems, where arrivals are  any generally distributed and time dependent processes, as long as types of successive arrivals are i.i.d.  Data for this more abstract model is very compact, it consists of the compatibility graph and the frequencies of the types.

This model was introduced in \cite{adan-kleiner-righter-weiss:18} where the following results were derived:
\begin{compactitem}[-]
\item
The matching is stable if and only if there are enough compatible goods for every subset of agent types (see Definition \ref{def.stabilitycond}).
\item
Under stability there is almost surely a unique matching for the bi-infinite sequence of agents and goods (see Uniqueness  Theorem \ref{thm.uniqueness}).
\item
Exchanging the positions of each matched pair in the bi-infinite sequence results in an i.i.d sequence, in which the same matches are now directed FCFS in reversed time (see Reversibility Theorem \ref{thm.reversibility}).
\item
The stationary distribution of unmatched agents at every stage along the sequence is derived up to a normalizing constant (see Product Form Theorem \ref{thm.stability}).
\end{compactitem}
Several questions were left unanswered in \cite{adan-kleiner-righter-weiss:18} and motivate the current paper, our main results are:
\begin{compactitem}[-]
\item
We find a detailed Markov description of the matching process, obtain its stationary distribution which is of a very simple intuitive form, and find a simple characterization of the possible states (see Section \ref{sec.detailed})
\item 
We calculate the normalizing constant which was not obtained in  \cite{adan-kleiner-righter-weiss:18} (see Section \ref{sec.constant}).
\item
We calculate the matching rates i.e. long range frequencies of $(s_j,c_i)\in \G$ matches, and the fraction of unmatched goods of each type (see Section \ref{sec.matchingrates})
\item
We obtain the distributions of delays, i.e. the distances between matched agent and good pairs, and derive explicit expressions for their generating functions and first and second moments (see Section \ref{sec.delays}).
\item
We show that in heavy traffic, the directed bipartite matching model converges in a certain sense to the symmetric bipartite matching model of 
 \cite{caldentey-kaplan-weiss:09,adan-weiss:12,adan-busic-mairesse-weiss:17,moyal-busic-mairesse:17} 
 (see Section \ref{sec.limits}).
\end{compactitem}

Before outlining our results we discuss some related literature:    The idea of studying parallel service systems by looking at FCFS matching of an infinite discrete sequence of customers (agents) and an infinite  discrete sequence of servers (goods) was first suggested in \cite{caldentey-kaplan-weiss:09},  motivated by a public housing allocation application \cite{kaplan:88}.  
In this bipartite matching formulation customers and servers play completely symmetric roles.  Several   examples,  with simple  compatibility graphs were analyzed in this paper.  A necessary and sufficient condition for stability and a product form stationary distribution  for the symmetric bipartite matching  model were derived in \cite{adan-weiss:12}.  The product form stationary distribution was derived by partial balance, similar to   
 \cite{visschers-adan-weiss:12,adan-weiss:14}.  Most important, the stationary distribution was then  used to derive expressions for the matching rates, i.e. the fraction of customers of each type served by each type of server (see Definition \ref{def.matchingrates} Section \ref{sec.matchingrates}).  These were considered quite intractable and are essential in studying the performance of FCFS parallel service systems.  
A  deeper analysis of the  symmetric bipartite matching  model  followed in  \cite{adan-busic-mairesse-weiss:17}, where three theoretical results were derived:  A uniqueness result analogous to Theorem \ref{thm.uniqueness}, a reversibility result analogous to Theorem \ref{thm.reversibility}, and detailed Markov descriptions of the matching process, with simple multi-Bernoulli stationary distributions, similar to Theorem \ref{thm.detailed} derived in this paper. In addition to the matching rates, another important performance measure was derived,  the distributions of link lags, i.e. the lags between matched customer server pairs (compare to Section \ref{sec.delays} here).

In a recent paper by Moyal, Busic and Mairesse \cite{moyal-busic-mairesse:17}, a matching model for a general undirected graph is formulated and analyzed:  there is a single sequence of items of several types, and a general compatibility graph,  and each item in the sequence is matched to an earlier compatible item if such exists, or it remains unmatched until it can be matched to a later item in the sequence.  For i.i.d. types, under FCFS matching they derive a condition for stability, they prove a uniqueness result and a reversibility result analogous to Theorems \ref{thm.uniqueness} and \ref{thm.reversibility}, and derive stationary product form distributions, up to a normalizing constant.  In the current version of the paper there is no calculation of the normalizing constant, or of the matching rates.   
Moyal et al. \cite{moyal-busic-mairesse:17} also note that if the compatibility graph is a bipartite graph, then the condition for stability is always violated.  We explain this in Section \ref{sec.limits}.

 Several papers that consider parallel service systems with  customers and servers of several types and bipartite compatibility graph are closely related to matching models.  
The stationary distribution of the parallel service system with Poisson arrivals and exponential service times, derived in \cite{adan-weiss:14} converges to the bipartite matching model of 
\cite{adan-weiss:12} is heavy traffic.  In \cite{gardner-etal:16} Gardner et. al. introduce another kind of FCFS policy, called a redundancy policy:  each server maintains its own FCFS queue, and customers make independent copies of their job, and join the queues of all the compatible servers, so that they may be served simultaneously by more than one server.  The customer then departs when his service at the first of the compatible servers is complete.  This is also called cancel-on-complete in \cite{ayesta-bodas-verloop:18}, to distinguish it from the model of \cite{adan-weiss:14} which is referred to as cancel-on-start.
  If  service times are exponential, then simultaneous service until first completion incurs no loss, and with Poisson arrivals the distribution of the queue lengths is shown to be of product form.  Kleiner et al. \cite{adan-kleiner-righter-weiss:18} compare the policies of \cite{adan-weiss:14} and \cite{gardner-etal:16} and discuss their relations to our FCFS directed bipartite matching model.  Surprisingly Kleiner \cite{adan-kleiner-righter-weiss:18} finds that the stationary Markov chain that describes the redundancy service system is the same as the one describing the unmatched agents in our model.  Further research on this type of model is surveyed in \cite{gardner-righter:19}.
  
In \cite{adan-boon-weiss:19} a heuristic based on calculation of matching rates for the symmetric FCFS bipartite matching model is used in the design of parallel service systems under many server scaling.  A similar design heuristic based on the calculations of matching rates for the directed FCFS matching model 
may be useful for design of two sided systems such as ride sharing etc.

We now briefly describe the results in the following sections.  In Section \ref{sec.previous} we define the directed FCFS bipartite matching model and summarize the results of \cite{adan-kleiner-righter-weiss:18}, which form the basis for our analysis in this paper.   We name this model a directed FCFS matching model to distinguish it from the symmetric FCFS matching model of \cite{caldentey-kaplan-weiss:09,adan-weiss:12,adan-busic-mairesse-weiss:17}:  in the symmetric model goods can be matched both to earlier and to later agents, while in the current model goods do not wait, and can only be matched to agents in the past direction.

In Section \ref{sec.detailed} we  consider the following detailed Markov chain that describes the evolution of the directed FCFS matching:  We perform matching of all goods up to time $N$ (i.e. position $N$ in the sequence),  and with each match we exchange the positions of the matched agent and good.  We then define the state as the list of items in the sequence, starting with the first unmatched agent, and ending at the $N$th item.  The list includes unmatched agents, unmatched and lost gods, and matched and exchanged agents and goods.  We show in Theorem \ref{thm.detailed} that the probability of such a list is simply the product of the frequencies of the items, multiplied by a normalizing constant.  This simplicity is deceptive, since it hides dependence on the constraints of what states are actually 
 admissible (i.e. states that occur with positive probability).  We next obtain a simple characterization of admissible states in Theorem \ref{thm.characterization}.  The results in this section are analogous to similar results in \cite{adan-busic-mairesse-weiss:17} and \cite{moyal-busic-mairesse:17}.

In Section \ref{sec.constant}  we consider a more concise Markov chain associated our model.  We again consider making all the matches up to a discrete time $N$,  and  
list the agents that remain unmatched.  The list of these unmatched agents up to $N$ is a Markov chain and its stationary product form distribution was given in  \cite{adan-kleiner-righter-weiss:18}, following the derivation in \cite{gardner-etal:16}.   A more concise Markov chain includes just the identities of which types  appear in the list of unmatched agents, in their order of appearance in the list, and the count of how many additional unmatched agents are between them.  This is also a Markov chain, and we derive its stationary distribution in Theorem \ref{thm.succint}.  We use this more tractable Markov chain to calculate the normalizing constant, which is also the normalizing constant for the detailed process and the process of  \cite{adan-kleiner-righter-weiss:18}.

In Section \ref{sec.matchingrates}  we derive expressions for the matching rates (Definition \ref{def.matchingrates}), i.e. the long range fraction of matches for each pair $s_j,c_i\in \G$, and the fraction of goods of each type that remain unmatched.   It is conjectured that the calculation of the matching rates (as well as the normalizing constant) is $\sharp P$.  Nevertheless, the calculation is straight forward, and we list an algorithm to perform it.  

In Section \ref{sec.delays}  we derive the distribution of delays, i.e. for each $s_j,c_i\in \G$, the distance in the sequence between the agent and the matched good.  We find that these distributions are a mixture of convolutions of independent geometric random variables.  We obtain explicit expressions for the generating functions and for the mean and variance of these delays. We list an enhancement to the algorithm of Section \ref{sec.matchingrates} to calculate these means and variances.  These can be used to obtain waiting times in the corresponding queueing models.

In Section \ref{sec.limits} we show that in heavy traffic the directed FCFS bipartite matching model converges to the symmetric matching model of  \cite{caldentey-kaplan-weiss:09,adan-weiss:12,adan-busic-mairesse-weiss:17}.  We also obtain the limiting matching rates in light traffic.

In Section \ref{sec.example} we illustrate our results for a simple example.  We demonstrate the advantage of pooling service by bipartite matching over discriminating service.


\section{The Model and Previous Results}
\label{sec.previous}

We consider a single infinite sequence of agents  and goods, which is generated as follows:  each successive item in the list is an agent with probability $\frac{\rlambda}{\rlambda + \rmu}$, or a good with probability $\frac{\rmu}{\rlambda + \rmu}$.  If it is an agent its type is $c_i$ with probability $\alpha_{c_i}$, and if it is a good, its type is $s_j$ with probability $\beta_{s_j}$, so that in summary, the elements of the sequence are i.i.d. $(z^n)_{n=\ldots,-2,-1,0,1,2,\ldots}$, and 
\begin{equation}
\label{eqn.iid}
P(z^n=c_i)= \frac{\rlambda}{\rlambda + \rmu}\alpha_{c_i}=\frac{\lambda_{c_i}}{\rlambda + \rmu},\quad P(z^n=s_j)=\frac{\rmu}{\rlambda + \rmu}\beta_{s_j}= \frac{\mu_{s_j}}{\rlambda + \rmu},   \qquad c_i\in \C,\; s_j\in \S,\;n\in \mathbb{Z}.
\end{equation}
where $\lambda_{c_i},\,\mu_{s_j}$ are as defined above,  and $\G$ is the bipartite compatibility graph.
We define the matching as follows:
\begin{definition}[Directed FCFS matching]
We say that agents and goods are matched according to directed FCFS matching if the following holds:  $z^n=s_j$ is matched with $z^m=c_i$ if and only if:
\begin{compactenum}[(i) ]
\item
$n>m$, 
\item
$(s_j,c_i)\in\G$, 
\item
For all  $k<m$ such that $z^k=c_{i'}$ and $(z_j,c_{i'})\in\G$, there exists  $z^l$ with $l<n$, such that $z^l$ is matched with $z^k$.
\end{compactenum}
If $z^n=s_j$ has no match that satisfies these conditions then $z^n$ remains unmatched.
\end{definition}

This matching model was introduced in \cite{adan-kleiner-righter-weiss:18}, and the following definitions and three theorems are quoted from that paper.

\paragraph{Notation}
We let $\S(c_i)$ denote the subset of good types compatible with $c_i$, and $\C(s_j)$ denote 
the subset of agent types compatible with $s_j$.   For  $C \subset \C$, $S\subset \S$ we let $\S(C) =\bigcup_{c_i\in C} \S(c_i)$, $\C(S) =\bigcup_{s_j\in S} \C(s_j)$,  and denote by $\U(S) = \overline{\C(\overline{S})}$ those agent types that are compatible only with good types in $S$.  We also let $\alpha_C=\sum_{c_i\in C} \alpha_{c_i}$, $\beta_S=\sum_{s_j\in S} \beta_{s_j}$ and $\lambda_C=\sum_{c_i\in C} \lambda_{c_i}$, $\mu_S=\sum_{s_j\in S} \mu_{s_j}$.

 \begin{definition}[Stability condition]
 \label{def.stabilitycond}
We say that the stability condition holds if for all $C \subseteq \C$, $S\subseteq \S$ the following holds:
\begin{equation}
\label{eqn.stability}
\lambda_C < \mu_{\S(C)}
\end{equation}
\end{definition}

The following Markov chain was defined in \cite{adan-kleiner-righter-weiss:18}:  Assume all the possible directed FCFS matches of goods up to $z^N$ have been made.  Define $X(N)$ as the ordered sequence $c^1,\ldots,c^L$, such that 
$c^l = z^{n_l}$, $n_1<\cdots<n_L \le N$ are the only unmatched agents remaining, ordered as in the sequence.  

The states of this Markov chain are described by finite sequences of agent types.  
This is a countable state Markov chain (the set of finite sequences of agent types is a countable set, even if the set of types is infinite).  Its transition probabilities are:
\begin{eqnarray}
\label{eqn.transitionsX}
P(X(N+1)=c^1,\ldots,c^L | X(N)=c^1,\ldots,c^L) &=& 
 \frac{\mu_{\overline{\S}(\{c^1,\ldots,c^L\})}}{\rlambda + \rmu}, \nonumber \\
P(X(N+1)=c^1,\ldots,c^L,c^{L+1} | X(N)=c^1,\ldots,c^L) &=&
\frac{\lambda_{c^{L+1}}}{\rlambda + \rmu}, \\
P(X(N+1)=c^1,\ldots,c^{i-1},c^{i+1},\ldots,c^L | X(N)=c^1,\ldots,c^L) &=&
\frac{\mu_{\overline{\S}(\{c^1,\ldots,c^{i-1}\})\cap \S(c^i)}}{\rlambda + \rmu},  \nonumber
\end{eqnarray}

It turns out that this Markov chain is the jump chain of the redundancy service model studied by \cite{gardner-etal:16}, where the following is proved.
\begin{theorem}[Product Form Theorem]
\label{thm.stability}
The Markov chain $X(N)$ is ergodic if and only if the stability condition (\ref{eqn.stability}) holds.  Its stationary distribution is given, up to a normalizing constant, by:
\begin{equation}
\label{eqn.stationaryX}
\pi_X(c^1,\ldots,c^L) \propto 
\prod_{\ell=1}^L \frac{\lambda_{c^\ell}}{\mu_{\S(\{c^1,\ldots,c^\ell\})}}.
\end{equation}
\end{theorem}

It is clear that for any sequence $z^0,z^1,\ldots,$ the directed FCFS matching is unique.  The corresponding Markov chain is $X(N),\,N=0,1,\ldots$ starting from no unmatched agents, i.e. $X(0)=\emptyset$.  However, If we prepend to the sequence for $n=0,1,\ldots$ a $z^{-1}$ and look at the directed FCFS matching for  $z^{-1},z^0,z^1,\ldots$ we usually will get a  different matching, i.e. the process starting from $X(-1)=\emptyset$ will be different from that starting with 
$X(0)=\emptyset$.   The following somewhat surprising theorem is similar to one proved for the undirected FCFS matching of a sequence of agents to an independent sequence of goods in \cite{adan-busic-mairesse-weiss:17}.
\begin{theorem}[Uniqueness Theorem]
\label{thm.uniqueness}
Let $\ldots,z^{-1},z^0,z^1,\ldots$ be a  sequence of agent  and good types chosen i.i.d. as above.  If the stability condition (\ref{eqn.stability}) holds then almost surely there  exists a directed FCFS matching of goods to cover all the agents, and this matching is unique.
\end{theorem}
The proof is given in \cite{adan-kleiner-righter-weiss:18}.  The main steps are to show forward coupling, then to show backward coupling, and then to conclude that a Loynes construction will yield the unique matching almost surely.   A similar uniqueness result is also proved in \cite{moyal-busic-mairesse:17}.

For the next result, which is  also analogous to a similar result in  \cite{adan-busic-mairesse-weiss:17}, and is proven in \cite{adan-kleiner-righter-weiss:18}, we define the following {\em exchange transformation}:
\begin{definition}[Exchange Transformation]
For the  directed  FCFS matching, after all possible matches were made on $\ldots,z^{-1},z^0,z^1,\ldots$, we define the exchanged sequence $\ldots,\tz^{-1},\tz^0,\tz^1,\ldots$ as follows:  If $z^m=c_i$ was matched to $z^n=s_j$, where $m<n$, then in the exchanged sequence we will have $\tz^m=s_j$, $\tz^n=c_i$.  If $z^n=s_j$ was unmatched, then $\tz^n=z^n$.
Throughout the paper we use $\tc,\ts$ to denote items that were matched and exchanged.
\end{definition}

As a consequence of the stability and the unique matching over $\mathbb{Z}$, we obtain:

\begin{theorem}[Reversiblity Theorem]
\label{thm.reversibility}  
Assume the stability condition (\ref{eqn.stability}) holds.
The sequence  $\ldots,\tz^{-1},\tz^0,\tz^1,\ldots$ obtained from the sequence 
$\ldots,z^{-1},z^0,z^1,\ldots$ by the exchange transformation is an i.i.d. sequence.  The unique directed matches for the new sequence performed in reversed time, result in exactly the reversed matches of the original sequence, almost surely.
\end{theorem}


\section{Detailed Markov Chains with Multi-Bernoulli Stationary Distributions}
\label{sec.detailed}
We now define two additional Markov chains associated with the directed FCFS infinite bipartite matching model.    These processes describe the evolution of matching and exchanging items in the sequence $(z^n)_{-\infty<n<\infty}$.
We define {\em detailed goods matching process} $U(N)$,  to describe
 matching and exchanging all  goods up to (and including) $N$ with  agents that appear earlier in the sequence (with some goods left unmatched).
We define {\em the detailed agents matching process} $W(N)$, to describe matching and exchanging  all  agents up to (and including) $N$ with goods that appear later in the sequence.
For $U(N)$  we perform directed FCFS matches and exchanges of  all compatible pairs $z^m=c_i,z^n=s_j,\, m<n \le N$,   while for $W(N)$   we perform directed FCFS matches and exchanges of all compatible  pairs $z^m=c_i,z^n=s_j,\, m<n$, for all agents $m\le N$, including  some $z^n=s_j$ for which $n>N$.

Specifically,  $U(N)$ is defined as follows: if $X(N)=\emptyset$, i.e. all agents appearing up to $N$ have been matched to goods appearing up to $N$, then $U(N)=\emptyset$.  Otherwise, let 
$\underline{N}$ be the location in the sequence of the first unmatched agent,  then 
$U(N)=u=(u^1,\ldots,u^L)$ where $L=N-\underline{N}+1$, and  $U(N)$ records  the elements in the sequence from $\underline{N}$ to $N$,  after matching and exchanging of goods to earlier agents, and also exchanging the remaining unmatched goods up to $N$ with themselves.
The elements of $u$ are then of types $c_i$ or $\tc_i$ or $\ts_j$ and are  constructed   as follows:
\begin{compactitem}[-]
\item
$u^1=c_i$ where $c_i=z^{\underline{N}}$ is the type of the first unmatched agent,
\item
$u^l=c_{i'}$ if $z^{\underline{N}+l-1}=c_{i'}$ is unmatched, 
\item
$u^l=\ts_j$ if $z^{\underline{N}+l-1}=s_j$  remains  unmatched (it is exchanged with itself),  
\item
$u^l=\tc_{i''}$ if $z^{\underline{N}+l-1}$ has originally been a good, and has been matched and exchanged with an agent of type $c_{i''}$ that appeared earlier in the sequence,
\item
$u^l=\ts_{j'}$ if $z^{\underline{N}+l-1}$ has originally been an agent, and has been matched and exchanged with a good of type $s_{j'}$ that has appeared later in the sequence, in a position $\le N$.
\end{compactitem}  

Similarly, $W(N)$ is defined as follows: if $X(N)=\emptyset$, i.e. all agents appearing up to $N$ have been matched to goods appearing up to $N$, then $W(N)=\emptyset$.  Otherwise, let 
$\overline{N}$ be the location in the sequence of the last matched and exchanged good (we have $\overline{N}>N$ since $X(N)\ne \emptyset$),  then 
$W(N)=w=(w^1,\ldots,w^K)$ where $K=\overline{N}-N$, and  $W(N)$ records  the elements in the sequence from  $N+1$ to $\overline{N}$,  after matching and exchanging of agents from locations up  to $N$ with later goods.
The elements of $w$ are  of types $c_i$ or $\tc_i$ or $s_j$ and are  constructed   as follows:
\begin{compactitem}[-]
\item
$w^K=\tc_i$ where $z^{\overline{N}+K}$ has been a good that has been matched and exchanged with an agent of type $c_i$ that appeared in the sequence in a position $\le N$
\item
$w^l=c_{i'}$ if $c_{i'}=z^{N+l}$ is as yet unmatched, because only agents up to $N$ have been matched.
\item
$w^l=s_j$ if $s_j=z^{N+l}$  has not been matched to an agent appearing up to $N$,  
\item
$w^l=\tc_{i'}$ if $z^{N+l}$ has originally been a good, and has been matched and exchanged with an agent of type $c_{i'}$ that appeared in the sequence in position $\le N$\end{compactitem}  

We now use the reversibility theorem.  
By the reversibility Theorem \ref{thm.reversibility} under the condition (\ref{eqn.stability}),  almost surely, if we take the sequence $(z^n)_{-\infty<n<\infty}$ and do all the directed FCFS matches and exchanges we obtain the sequence $(\tz^n)_{-\infty<n<\infty}$ which is again i.i.d. and if we do directed FCFS matches and exchanges in reversed time on the sequence $(\tz^n)_{-\infty<n<\infty}$  we will retrieve the original sequence $(z^n)_{-\infty<n<\infty}$.  In fact, the pairs that are matched are the same in both sequences.

We  consider the sequence $(\tz^n)_{-\infty<n<\infty}$, and define  goods and agents detailed matching processes for it in reversed time:   $\tilde{U}(N+1)$, $\tilde{W}(N+1)$ are obtained when we do matching for the  sequence $(\tz^n)_{-\infty<n<\infty}$, 
where for $\tilde{U}(N+1)$  we perform directed FCFS matches and exchanges in reversed time, of  pairs $\tz^n=\ts_j, \,\tz^m=\tc_i, m>n$, for all goods in positions $n > N$, while for $\tilde{W}(N+1)$   we perform directed FCFS matches and exchanges in reversed time of  pairs $\tz^n=\ts_j, \,\tz^m=\tc_i, m>n$, for all agents in positions $m> N$.

For any vector $u$ consisting of items $c,\tc,s,\ts$ we denote by $\tu$ the exchanged vector in which we exchange between $c,\tc$ and between $s,\ts$.  We denote by $\au$ the reversed vector which is $u$ written in reversed order.  We will also use the notation $\atu$ for exchanging and reversing.

\begin{proposition}
\label{thm.same}
For almost all  $(z^n)_{-\infty<n<\infty}$ and corresponding  $(\tz^n)_{-\infty<n<\infty}$, and  for  $-\infty<N<\infty$ we have: 
\begin{equation}
\label{eqn.same1}
\mbox{If }\quad  
\begin{array}{l}
U(N)=u=(u^1,\ldots,u^L) \\
U(N+1) = u' = ({u'}^1,\ldots,{u'}^{L'})
\end{array}
\quad\mbox{ then } \quad
\begin{array}{l}
\tilde{W}(N+2)={\au}'=({u'}^{L'},\ldots,{u'}^1),\\
\tilde{W}(N+1)=\au=(u^L,\ldots,u^1).
\end{array}
\end{equation}
and similarly,
\begin{equation}
\label{eqn.same2}
\mbox{If } \quad 
\begin{array}{l}
W(N)=w=(w^1,\ldots,w^K) \\
W(N+1)=w'=({w'}^1,\ldots,{w'}^{K'}) 
\end{array}
\quad \mbox{ then }\quad 
\begin{array}{l}
\tilde{U}(N+2)={\aw}'=({w'}^{K'},\ldots,{w'}^1), \\
\tilde{U}(N+1)=\aw=(w^K,\ldots,w^1).
\end{array}
\end{equation}
\end{proposition}
\begin{proof}
By definition, $U(N)$ will have all goods in the sequence $z^n$ up to $N$ matched and exchanged, and will list the resulting items from $\underline{N}$ to $N$, while all items in positions $>N$ are still unmatched and not exchanged.  On the other hand, $\tilde{W}(N+1)$ will have all the agents  in the exchanged sequence $\tz^n$ from $N+1$ to $\infty$ matched and exchanged back to their values in $z^n$, and list the items from $N$ to $\underline{N}$ in reversed order.  
In particular, the first item in $U(N)$ is  the earliest unmatched agent $z^{\underline{N}}=c_i$, and it could only be matched later with $z^m=s_j$ which is in position $m>N$, while  in $\tilde{W}(N+1)$ $\tz^m=\tc_i$ is matched and exchanged back with $\tz^{\underline{N}}=\ts_j$.  Similarly any unmatched agent $z^n=c_i,\,n<N$ in $U(N)$ has a match $z^m$ with $m>N$, hence $\tz^n=\ts_j$, and $\tz^m=\tc_i$, and in $\tilde{W}(N+1)$ they will be matched and exchanged back.    All other elements of $U(N)$ are exchanged $\tc$ or $\tz$ which are in their original positions in $\tilde{W}(N+1)$.
Hence $\tilde{W}(N+1)$ lists the same items as $U(N)$ in reversed order.
By exactly the same argument,  $\tilde{W}(N+2)$ lists the same items as $U(N+1)$ in reversed order.
The statement  (\ref{eqn.same2}) can be proved similarly, and in fact it follows directly from the reversibility Theorem \ref{thm.reversibility}.
\qed
\end{proof}

\begin{proposition}
\label{thm.conditionals}
For the stationary versions of the detailed Markov process $U(N)$ the reversed transition probabilities are given by: 
\begin{equation}
\label{eqn.conditionals}
\begin{array}{l}
P\big(U(N+1)=u'\big)\,P\big(U(N)=u|U(N+1)=u'\big)   \\
=P\big(W(N)=\overleftarrow{\tilde{u}'} \big)\, P\big(W(N+1)=\atu | W(N)=\overleftarrow{\tilde{u}'} \big). 
\end{array}
\end{equation}
\end{proposition}
\begin{proof}
\begin{equation}
\begin{array}{ll}
P\big(U(N+1)=u'\big)\,P\big(U(N)=u|U(N+1)=u'\big)  &\\ 
\;= P\big(U(N)=u \, , \, U(N+1)=u'\big)  &  \\
\;=  P\big(\tW(N+1)=\overleftarrow{u} \, , \,\tW(N+2)=\au'  \big) & \\
\;= P\big(\tW(N+2)=\au' \big)\, P\big(\tW(N+1)=\overleftarrow{u} | \tW(N+2)=\au' \big) & \\
\;= P\big(W(N)=\atu' \big)\, P\big(W(N+1)=\atu | W(N)=\atu' \big). 
\end{array}
\end{equation}
The first and third equalities are by definition of conditional probability.  
The second equality holds for the stationary processes by Proposition \ref{thm.same}, equation (\ref{eqn.same1}).  The last equality follows from the reversibility Theorem \ref{thm.reversibility}, since the Markov chain $W(N)$ driven by the sequence $z^n$  has the same stationary distribution and the same transition probabilities as the Markov chain $\tW(N)$ driven by the sequence $
\tz^n$ in reversed time.
\qed  \end{proof}

\begin{theorem}[Multi-Bernoulli Form]
\label{thm.detailed}
If the stability conditions (\ref{def.stabilitycond}) holds, then the stationary distribution of the detailed Markov chains is given by:
\begin{equation}
\label{eqn.detailedstationary}
\begin{array}{l}
\displaystyle  \pi_U(u_1,\ldots,u_L)=\pi_W(u_1,\ldots,u_L)  \\  
\displaystyle = B 
\left(\frac{\rlambda}{\rlambda + \rmu}\right)^{\sharp \mbox{agents}}
\left(\frac{\rmu}{\rlambda + \rmu}\right)^{\sharp \mbox{goods}}
\prod_{c_i\in\C} \alpha_{c_i}^{\sharp c_i + \sharp \tc_i}
\prod_{s_j\in\S} \beta_{s_j}^{\sharp s_j + \sharp \ts_j},
\end{array}
\end{equation}
where $B$ is the normalizing constant given by (\ref{eqn.constant}), and we are counting the number of agents and of goods overall, and of each type separately, in $u_1,\ldots,u_L$.
\end{theorem}
\begin{proof}
That $\pi_U(u_1,\ldots,u_L)=\pi_W(u_1,\ldots,u_L)$ follows from Proposition \ref{thm.same}. 
By definition   
$\pi_U(\emptyset)=\pi_W(\emptyset)=\pi_X(\emptyset)$.  We will show that  $\pi_X(\emptyset)=B$ and show that  $B$ is given by (\ref{eqn.constant}) in Section \ref{sec.constant}.
 It remains to prove (\ref{eqn.detailedstationary}) for $\pi_U$.  For a stationary Markov chain this is equivalent to verifying that:
\begin{equation}
\label{eqn.kelly}
\pi_U(u)P\big(U(N+1)=u'|U(N)=u\big) = \pi_U(u')P\big(U(N)=u|U(N+1)=u'\big).
\end{equation}
We have found in Proposition \ref{thm.conditionals} that the reversed transition  probabilities for $U$ are given by transition probabilities of $W$.   Also, we note that in the form of equation (\ref{eqn.detailedstationary}), $\pi_U(u)=\pi_U(\au)=\pi_U(\tu)=\pi_U(\atu)$, with similar equalities for $W$.  
Therefore to prove (\ref{eqn.kelly}) we will verify that the expressions in (\ref{eqn.detailedstationary}) satisfy:
\begin{equation}
\label{eqn.kelly1}
\frac{\pi_U(u')}{\pi_U(u)} =  
\frac{ P\big(U(N+1)=u'|U(N)=u\big) }
{ P\big(W(N+1)=\atu | W(N)=\overleftarrow{\tilde{u}'}\big) }.
\end{equation}
and to do so we need to calculate  transition probabilities for $U$ and for $W$ from $N$ to $N+1$.  These transition probabilities hinge on the identity of $x=z^{N+1}$ and the corresponding  $\tx=\tz^{N+1}$.   We start from $U(N)=u=(u_1,\ldots,u_L)$.
There are four cases:

\subsubsection*{Case (i):  $x = c_i$}
If $U(N)=u=(u_1,\ldots,u_L)$  and $z^{N+1}=c_i$, then $U(N+1)=u'=(u_1,\ldots,u_L,c_i)$.
The probability of the transition 
$u\to u'$ is $\frac{\rlambda}{\rlambda + \rmu} \alpha_{c_i}$.

In that case $W(N)=\overleftarrow{\tilde{u}'}=(\tc_i,\tu_L,\ldots,\tu_1)$.  This means that the  $z^{N+1}$ has already been matched and exchanged, and is now $\tz^{N+1}=\tx=\tc_i$.
Therefore all agents up to $N+1$ have been matched, and $W(N+1)=\atu = (\tu_L,\ldots,\tu_1)$.  The probability of the transition:
$\overleftarrow{\tilde{u}'} \to \atu$ is 1.

We now check:
\begin{eqnarray*}
&& \frac{ \pi_U(u_1,\ldots,u_L,c_i) }{ \pi_U(u_1,\ldots,u_L) } = \frac{\rlambda}{\rlambda + \rmu} \alpha_{c_i}, \\
&& \frac{ P\big(  U(N+1)=(u_1,\ldots,u_L,c_i) \,|\, U(N)= (u_1,\ldots,u_L)  \big) }
{ P\big(  W(N+1)= (\tu_L,\ldots,\tu_1) \,|\, W(N) = (\tc_i,\tu_L,\ldots,\tu_1) \big) } =\frac{\rlambda}{\rlambda + \rmu} \alpha_{c_i}.
\end{eqnarray*}

\subsubsection*{Case (ii):  $x = s_j$, and $s_j$ is incompatible with the unmatched agents in 
$(u_1,\ldots,u_L)$}

If $U(N)=u=(u_1,\ldots,u_L)$  and $z^{N+1}=s_j$ and $s_j$ is incompatible with any of the 
unmatched agents in $(u_1,\ldots,u_L)$,   
then $s_j$ remains unmatched and is replaced by $\ts_j$ (matched with itself), so that $U(N+1)=u'=(u_1,\ldots,u_L,\ts_j)$.
The probability of the transition 
$u\to u'$ is $\frac{\rmu}{\rlambda + \rmu} \beta{s_j}$.

In that case $W(N)=\overleftarrow{\tilde{u}'}=(s_j,\tu_L,\ldots,\tu_1)$.  This means that the $N+1$ element is $z^{N+1}=s_j$.  Since all agents prior to $N$ have been matched, $s_j$ will remain unmatched, and in the transition from $W(N)$ to $W(N+1)$ there will be no change in $\tu_L,\ldots,\tu_1)$, and so $W(N+1)=\tu_L,\ldots,\tu_1)=\atu$.  The probability of this transition is again 1.

We now check:
\begin{eqnarray*}
&& \frac{ \pi_U(u_1,\ldots,u_L,\ts_j) }{ \pi_U(u_1,\ldots,u_L) } = \frac{\rmu}{\rlambda + \rmu} \beta_{s_j}, \\
&& \frac{ P\big(  U(N+1)=(u_1,\ldots,u_L,\ts_j) \,|\, U(N)= (u_1,\ldots,u_L)  \big) }
{ P\big(  W(N+1)= (\tu_L,\ldots,\tu_1) \,|\, W(N) = (s_j,\tu_L,\ldots,\tu_1) \big) } 
= \frac{\rmu}{\rlambda + \rmu} \beta_{s_j}.
\end{eqnarray*}

\subsubsection*{Case (iii):  $x = s_j$, which is matched FCFS with $u_l = c_i$, where $l>1$}
If $U(N)=u=(u_1,\ldots,u_L)$  and $z^{N+1}=s_j$ and $s_j$ is compatible with 
$u_l = c_i$, and is incompatible with all unmatched agents among $u_1,\ldots,u_{l-1}$, then $s_j$ is matched and exchanged with $u_l$, so:  $U(N+1) =u'= (u_1,\ldots,u_{l-1},\ts_j,u_{l+1},\ldots,u_L,\tc_i)$.  The probability of this transition $u\to u'$ is  $\frac{\rmu}{\rlambda + \rmu} \beta{s_j}$.

In that case 
$W(N)=\overleftarrow{\tilde{u}'}=(c_i,\tu_L,\ldots,\tu_{l+1},s_j,\tu_{l-1}.\ldots,\tu_1)$.  
Because $z^{N+1}=c_i$, it needs to be matched and exchanged in the transition from 
$W(N)$ to $W(N+1)$.  Clearly it can be matched and exchanged with $s_j$ which is between 
$\tu_{l+1},\tu_{l-1}$.  We claim that $c_i$ cannot be matched to any of $\tu_L,\ldots,\tu_{l+1}$.  Recall that it was left unmatched in $U(N)$.  Assume some $\,l+1\le k \le L$, $\tu_k=s_{j'}$ is compatible with $c_i$.  in that case, $u_k=\ts_{j'}$ and there are two possibilities:  either $z^{\underline{N}+k-1}=s_{j'}$ could not be matched, and was exchanged to $\ts_{j'}$, or $z^{\underline{N}+k-1}=c_{i'}$ that was matched and exchanged with $s_{j'}$ that appeared later in the sequnce but before $N+1$.  But in either of these cases $s_{j'}$ should then have been matched and exchanged with $z^{\underline{N}+l-1}=c_i$, which is in contradiction to $u_l=c_i$ being unmatched in $U(N)$.  It follows that in this case in
the transition from $W(N)$ to $W(N+1)$, $c_i$ will be matched and exchanged with $s_j$ and $W(N+1) = (\tu_L,\ldots,\tu_{l+1},\tc_i,\tu_{l-1}.\ldots,\tu_1)$, which is $\atu$.  The probability of this transition is 1.

We now check:
\begin{eqnarray*}
&& \frac{ \pi_U(u_1,\ldots,u_{l-1},\ts_j,u_{l+1},\ldots,u_L,\tc_i) }
{ \pi_U(u_1,\ldots,u_{l-1},c_i,u_{l+1},\ldots,u_L) } 
= \frac{\rmu}{\rlambda + \rmu} \beta{s_j}, \\
&& \frac{ P\big(  U(N+1)=(u_1,\ldots,u_{l-1},\ts_j,u_{l+1},\ldots,u_L,\tc_i) \,|
\, U(N)= (u_1,\ldots,u_{l-1},c_i,u_{l+1},\ldots,u_L)  \big) }
{ P\big(  W(N+1)= (\tu_L,\ldots,\tu_{l+1},\tc_i,\tu_{l-1}.\ldots,\tu_1) \,|
\, W(N) = (c_i,\tu_L,\ldots,\tu_{l+1},s_j,\tu_{l-1}.\ldots,\tu_1) \big) } 
= \frac{\rmu}{\rlambda + \rmu} \beta_{s_j}.
\end{eqnarray*}

\subsubsection*{Case (iv):  $x = s_j$, which can matched FCFS to $u_1 = c_i$.}
In Cases (i)-(iii), the reversed transition was  $\atu' \to \atu$ with probability 1.  In Case (iv) this is no longer so, and we will need to calculate the probability of this transition.

Assume $U(N)=u=(c_i,u_2,\ldots,u_{l-1},c_{i'},u_{l+1},\ldots,u_L)$, where $u_2,\ldots,u_{l-1}$ contain no unmatched agents, and also it does not contain any goods compatible with $c_i$. Assume $z^{N+1}=s_j$, and $s_j$ is compatible with 
$u_1 = c_i$.
Then $s_j$ is matched and exchanged with $u_1$ (recall that $u_1$ is always an unmatched agent).  The resulting state then is $U(N+1)=u'=(c_{i'},u_{l+1},\ldots,u_L,\tc_i)$, because now $z^{\underline{N}+l-1}=u_l=c_{i'}$ is the first unmatched agent.
The  probability of this transition $u\to u'$ is  $\frac{\rmu}{\rlambda + \rmu} \beta{s_j}$.

In that case 
$W(N)=\overleftarrow{\tilde{u}'}=(c_i,\tu_L,\ldots,\tu_{l+1},\tc_{i'})$.  
Because $z^{N+1}=c_i$, it needs to be matched and exchanged in the transition from 
$W(N)$ to $W(N+1)$.  By the same argument as for Case (iii), $c_i$ is not compatible with any unmatched goods in $\tu_L,\ldots,\tu_{l+1}$.  We now consider the possibility that conditional on $W(N)=\overleftarrow{\tilde{u}'}$, the next transition will be to the state $W(N+1)=\atu=(\tu_L,\ldots,\tu_{l+1},\tc_{i'},\tu_{l-1},\ldots,\tu_2,\tc_i)$.  This will happen if and only if  $(z^{\overline{N}+1},\ldots,z^{\overline{N}+l-1}) = (\tu_{l-1},\ldots,\tu_2,s_j)$.  
Denote by $\theta$ the probability of this event,  then:
\begin{eqnarray*}
\theta &=& P(W(N+1)=\atu\,|\,W(N)=\overleftarrow{\tilde{u}'}) =
P\big((z^{\overline{N}+1},\ldots,z^{\overline{N}+l-1}) = (\tu_{l-1},\ldots,\tu_2,s_j)\big) \\
&=& \left(\frac{\rlambda}{\rlambda + \rmu}\right)^{\sharp \mbox{agents}}
\left(\frac{\rmu}{\rlambda + \rmu}\right)^{\sharp \mbox{goods}}
\prod_{c_i\in\C} \alpha_{c_i}^{\sharp c_i + \sharp \tc_i}
\prod_{s_j\in\S} \beta_{s_j}^{\sharp s_j + \sharp \ts_j}  \times 
\frac{\rmu}{\rlambda + \rmu} \beta{s_j},
\end{eqnarray*}
where we are counting the number of agents and of goods overall, and of each type separately, in $\tu_{l-1},\ldots,\tu_2$.

We now check:
\begin{eqnarray*}
&& \frac{ \pi_U(c_{i'},u_{l+1},\ldots,u_L,\tc_i) }
{ \pi_U(c_i,u_2,\ldots,u_{l-1},c_{i'},u_{l+1},\ldots,u_L) } 
=  \theta^{-1}  \frac{\rmu}{\rlambda + \rmu} \beta{s_j}, \\
&& \frac{ P\big(  U(N+1)=(c_{i'},u_{l+1},\ldots,u_L,\tc_i) \,|
\, U(N)= (c_i,u_2,\ldots,u_{l-1},c_{i'},u_{l+1},\ldots,u_L)  \big) }
{ P\big(  W(N+1)= (\tu_L,\ldots,\tu_{l+1},\tc_{i'},\tu_{l-1},\ldots,\tu_2,\tc_i) \,|
\, W(N) = (c_i,\tu_L,\ldots,\tu_{l+1},\tc_{i'}) \big) } 
=  \theta^{-1} \frac{\rmu}{\rlambda + \rmu} \beta_{s_j}.
\end{eqnarray*}
\qed  \end{proof}

The stationary distribution of the Markov chains $U(N)$ and $W(N)$ looks very simple:  It seems to be the distribution of a sequence of independent trials with outcomes in $\C\cup \S$ with fixed probabilities, we call that a multi-Bernoulli sequence.  However, this simplicity is very deceptive:  to say anything more about the distribution we need to know what are the possible states that may occur, and we need to find a way to count them.  The next theorem gives a characterization of the valid states of $U$ and $W$.

\begin{theorem}[Detailed State Cahracterization]
\label{thm.characterization}
A state $U(N)=(u_1,\ldots,u_L)$ has a positive probability of occurring under the stationary distribution if  and only if it  is $\emptyset$, or it starts with an unmatched agent, it consists of elements $c_i,\tc_i,\ts_j$, and for any $k<l$ if $u_k=c_i$ and $u_l=\ts_j$ then $(s_j,c_i)\not\in \G$.  

Similarly, a state $W(N)=(w_1,\ldots,w_L)$ has a positive probability of occurring under the stationary distribution if and only if it is $\emptyset$, or it ends  with an exchanged agent, it consists of elements $c_i,\tc_i,s_j$, and for any $k>l$ if $w_k=\tc_i$ and $w_l=s_j$ then $(s_j,c_i)\not\in \G$.  
\end{theorem}
\begin{proof}
We prove the theorem for $U$.  The result for $W$ follows from the reversibility and Proposition \ref{thm.same}.  

The necessity is easy to see.  By definition $(u_1,\ldots,u_L)$ consists only of  elements $c_i,\tc_i,\ts_j$.  Assume to the contrary, that  for some $k<l$ $u_k=c_i$ and $u_l=\ts_j\in \S(c_i)$.  If  $u_l=\ts_j$ then either $z^{\underline{N} + l -1} = s_j$ and $z^{\underline{N} + l -1}$ remained unmatched and was exchanged with itself, or   
$z^{\underline{N} + l -1} = c_{i'}$ and it was matched and exchanged with $z^{\underline{N} + m -1}=s_j$, for some $l<m\le N$.  But in either of these cases $z^{\underline{N} + k -1}=c_i$ should have been matched and exchanged with $s_j$, which is a contradiction.

More subtle is the proof that any $(u_1,\ldots,u_L)$ that satisfies the conditions can occur with positive probability as $U(N)$.   For such $u=(u_1,\ldots,u_L)$, we construct $\tu=(\tu_1,\ldots,\tu_L)$, and $w$ as follows:
\begin{compactenum}[(i)]
\item
Leave all  unmatched elements $u_m=c_i$ in place, i.e. let  $\tu_m=c_i$, this includes $\tu_1$ since $u_1$ is the first unmatched agent.
\item
Perform directed FCFS matching in reversed time on the elements $\tc_i$ and $\ts_j$ in $u$, and exchange their places, so that if $u_k= \tc_i$ is matched with $u_l=\ts_j$ (where $k>l$) we let  $\tu_k=s_j$ and $\tu_l=c_i$.   
\item
Leave all the remaining elements $u_m=\ts_j$ (if there are any) in place and exchange them, i.e.  let  $\tu_m=s_j$.
\item
Consider the remaining elements $u_{m_1}=\tc_{i_1},\ldots,u_{m_K}=\tc_{i_K}$ (if there are any), where $m_1<m_2\cdots < m_K$.
\item
For each $u_{m_k}=\tc_{i_k}$ let $\tu_{m_k}=s_{j_k}$ where you choose $s_{j_k}\in\S(c_{i_k})$.
\item
Append $(c_{i_1},\ldots,c_{i_K})$ to the constructed $(\tu_1,\ldots,\tu_L)$, to obtain
 $w=((c_{i_1},\ldots,c_{i_K},\tu_1,\ldots,\tu_L)$.
\end{compactenum}
Note that all elements $\tu_l,\,l=1,\ldots,L$ are well defined, and note that all elements of $w$ are either $c_i\in \C$ or $s_j\in \S$.

Assume that $U(N-K-L)=\emptyset$, and let $z^{N-K-L+k}=w_k$, $k=1,\ldots,L+K$.  
It is important to note that there is a positive probability for this event.  We now claim that $U(N)=u$.  We will  go through the steps of moving from 
$U(N-K-L)$ to $U(N)$, using directed FCFS, and matching and exchanging agent-good pairs.  

We  look first for matches for  $c_{i_1},\ldots,c_{i_K}$.   We start with $c_{i_1}$, the first element of $w$, i.e. $c_{i_1}=z^{\underline{N}-K-L+1}$.  We claim that it will be matched and exchanged with $\tu_{m_1}=s_{j_1}$.  By construction, $(c_{i_1},s_{j_1})$ are compatible, and $\tu_{m_1}=z^{\underline{N}-L+m_1}$ precedes all the other $\tu_{m_k}$ in the sequence $z^n$.  We need to show that if there is a 
$\tu_l=z^{\underline{N}-L+l} = s_j$, where $l<m_1$, then it must be incompatible with $c_{i_1}$.   
Assume to the contrary that such $\tu_l=z^{\underline{N}-L+l} = s_j$ with $l<m_1$ is compatible with $c_{i_1}$.  We examine $\tu_l=s_j$.  By construction of $\tu$, there are two possibilities:  
either for some $l'<l$, $u_{l'}=\ts_j$ was matched and exchanges with $u_{l}=\tc_{i}$ in step (ii) of the construction of $\tu$, or  else  $u_l=\ts_j$ was not matched in step (ii) and so was left in place, as $u_l=s_j$ in step (iii) of the construction.   The two possibilities are illustrated in the following line:
\[
\begin{array}{cccccccc}
&1&  & l & & m_1 & & L \\
\tu =&   (\tu_1,&\ldots&s_j,&\ldots,&s_{j_1},&\ldots&\tu_L) \\
u = &  (u_1,&\ldots&\ts_j,&\ldots,&\tc_{i_1},&\ldots&u_L) \\
  \end{array}
\quad , \qquad
\begin{array}{ccccccccccc}
&1& & l'  &  & l & & m_1 & & L \\
\tu =&   (\tu_1,&\ldots&c_i,&\ldots,&s_j,&\ldots,&s_{j_1},&\ldots&\tu_L) \\
u = &  (u_1,&\ldots&\ts_j,&\ldots,&\tc_i,&\ldots,&\tc_{i_1},&\ldots&u_L) \\
  \end{array}  
\]
We now get a contradiction.  Recall that we assume the $u_{m_1}=\tc_{i_1}$ was left unmatched in step (ii) and reached step (iv).  If $u_l=\ts_j$ was not exchanged in step (ii), then it would remain compatible with $u_{m_1}=\tc_{i_1}$, and should have been matched and exchanged with it in step (ii).   If $u_{l'}=\ts_j$ where $l'<l$ was matched and exchanged with $u_l=\tc_i$, in step (ii), then $u_{l'}$ should have actually have been exchanged with $u_{m_1}=\tc_{i_1}$ and not with $u_l=\tc_i$.

We  proceed by induction: if $w_1=c_{i_1},\ldots,w_{k-1}=c_{i_{k-1}}$ have been matched and exchanged with $\tu_{m_1}=s_{j_1},\ldots, \tu_{m_{k-1}}=s_{j_{k-1}}$, then $w_k=\tc_{i_k}$ cannot be matched to any of  $\tu_{m_1}=s_{j_1},\ldots, \tu_{m_{k-1}}=s_{j_{k-1}}$, and by the same argument as for $w_1$ it cannot be matched to any $\tu_l=s_j$ with $l<m_k$.  Hence it will be matched and exchanged with $\tu_{m_k}=s_{j_k}$.

We have then shown that by the time we reach $U(N)$, the first remaining unmatched agent in $\tu_1=u_1$, and all the elements in positions $m_1,\ldots,m_K$ are now occupied by $\tc_{i_1},\ldots,\tc_{i_K}$.   We note that  in all the steps of the construction of $\tu$ we were converting $\ts_j$ to $s_j$, but never changing the types of these goods.  By the assumption of the theorem, all $u_l=\ts_j$ are of types incompatible with all unmatched $u_k=c_i$, if $k<l$.   Therefore in step (i) all unmatched $u_l=c_i$ will  now be $\tu_l=c_i$, and will remain unmatched in $U(N)$.   So $U(N)$ coincides with $u$ in positions $m_1,\ldots,m_L$, and in those positions of $u$ that had elements $c_i$.  

The remaining positions in $u$ and $\tu$  are occupied by the agent-good pairs exchanged in step (ii) and the left over goods of step (iii).  We cIaim that if $u_k=\tc_i$ and $u_l=\ts_j$, $k>l$ were exchanged to $\tu_l=c_i$ and $\tu_k=s_j$ in step (ii), then in the directed FCFS matching leading to $U(N)$, the agent $\tu_l=c_i$ will be matched and exchanged with $\tu_k=s_j$ back to their original positions.   Assume there is a $\tu_m=s_{j'}$, such that $l<m<k$ and $(s_{j'},c_i)\in \G$. 
If $\tu_m=s_{j'}$ resulted from an exchange with $\tu_{m'}=c_{i'}$, then we must have had $l<m<k<m'$, since step (ii) did matching in reversed FCFS, and so $\tu_l=c_i$ did not match with  $\tu_m=s_{j'}$.  If $\tu_m=s_{j'}$ was obtained in step (iii) from $u_m=\ts_{j'}$ this would imply that it should have matched in step (ii) with $u_k=\tc_i$, which is a contradiction.
Hence we have the in $U(N)$ all pairs exchanged from their positions in $u$ in step (ii), would be exchanged back to their original positions in $U(N)$.  All that remains are goods remaining unmatched in step (iii) which will retain their positions in $u$ as in $\tu$ and in $U(N)$.
This completes the proof that $U(N)=u$.
\qed  \end{proof}


\section{A Concise Markov Chain, the Normalizing Constant, and a Sanity Check} 
\label{sec.constant}

The  Markov chain $U(N)$, and its reversal partner $W(N)$, contain  very detailed  information about  the evolution of the directed FCFS matching.  Much of this information is irrelevant, in particular the $\tc$ items and $\ts$ items in $U(N)$  refer to the past of the matching process, and have no influence on future matching. The only relevant part for the future is the ordered list of unmatched agents, which is given by the Markov chain $X(N)$,  introduced in \cite{adan-kleiner-righter-weiss:18}, where its stationary distribution (Theorem \ref{thm.stability}, equation (\ref{eqn.stationaryX})) is obtained using a partial balance argument from \cite{gardner-etal:16}.  We have briefly summarized results on $X(N)$ in Section \ref{sec.previous}.   
We note that the stationary distribution (\ref{eqn.stationaryX})) of $X(N)$ can be obtained directly from $U(N)$.  Later in this section (sub-section \ref{sec.sanity}) we illustrate such a derivation using $U(N)$.

The transitions of $X(N)$ are very simple:  given $X(N)=(c^1,\ldots,c^L)$ and $z^{N+1}$ the state $X(N+1)$ is determined uniquely.   On the other hand, the states are given by finite words of arbitrary length, which is awkward for calculations.  We next define a more concise Markov chain with simpler, finite dimensional states.

We define the Markov chain $Y(N)_{N\in\mathbb{Z}}$ whose states are of the form $Y(N)=y=(C_1,n_1,\ldots,C_k,n_k)$, where $k\le I$, $(C_1,\ldots,C_k)$ is a permutation of $C \subseteq \C$,  and $n_j$ are positive integers.  In pariticular, for $k=0$ the state is denoted as $Y(N)=\emptyset$ or $Y(N)=0$.
The states of $Y(N)_{N\in\mathbb{Z}}$  are derived from the  states of $X(N)_{N\in\mathbb{Z}}$ as follows:   If $X(N)=\emptyset$ then also $Y(N)=\emptyset$.
If $X(N)=(c^1,\ldots,c^L)$  then  $C_1,\ldots,C_k$ are the types of agents appearing in $x$, in the order in which they appear, and $n_j$ is the number of terms in $x$ between the first appearance of $C_j$ and the first appearance of $C_{j+1}$, for $j=1,\ldots,k-1$, and $n_k$  is the number of terms in $x$ from the first appearance of  $C_k$ up to the last term $c^L$.  The number of unmatched agents for $Y(N)=(C_1,n_1,\ldots,C_k,n_k)$ is $k+\sum_{j=1}^k n_j$.  Note that now, given $Y(N)$ and $z^{N+1}$ we can determine the next match, but the next state is not deterministic, several possible states $Y(N+1)$ may result.  This Markov chain is similar to a Markov chain used in \cite{visschers-adan-weiss:12,adan-weiss:14}.

\begin{theorem}
\label{thm.succint}
 If the stability condition (\ref{eqn.stability}) holds then the stationary distribution of $Y(N)_{N\in\mathbb{Z}}$  is given by:
 \begin{equation}
 \label{eqn.stationaryY}
 \pi_Y(C_1,n_1,\ldots,C_k,n_k) = 
B   \prod_{\ell=1}^k  \frac{\lambda_{C_\ell}}{\mu_{\S(\{C_1,\ldots,C_\ell\})}}
\left( \frac{\lambda_{\{C_1,\ldots,C_\ell\}}}{\mu_{\S(\{C_1,\ldots,C_\ell\})}} \right)^{n_\ell}
\end{equation}
where $B$ is a normalizing constant, which is also equal to the probability of a perfect match, i.e. $\pi_Y(\emptyset)=B$.
The value of $B$ is given by (\ref{eqn.constant}).
\end{theorem}
\begin{proof}
We note that  $Y(N)=(C_1,n_1,\ldots,C_k,n_k)$ includes all states $X(N)=(c^1,\ldots, c^L)$ where $L=k+\sum_{\ell=1}^k n_\ell$, where $C_\ell,\,\ell=1,\ldots,k$ appear at least once, in that order,
and where between the appearance of $C_\ell$ and $C_{\ell+1}$ there are any combinations of $n_\ell$ agents of types $C_1,\ldots,C_\ell$, hence by (\ref{eqn.stationaryX}):
\begin{eqnarray*}
 \pi_Y(C_1,n_1,\ldots,C_k,n_k) &=& B
 \prod_{\ell=1}^k  \frac{\lambda_{C_\ell}}{\mu_{\S(\{C_1,\ldots,C_\ell\})}}
 \left( \sum_{i_1+\cdots+i_\ell=n_\ell}  \frac{n_\ell !}{ i_1 ! \cdots i_\ell !} 
 \prod_{j=1}^\ell\left( \frac{\lambda_{C_j}}{\mu_{\S(\{C_1,\ldots,C_\ell\})}}\right)^{i_j} 
 \right)  \\
 &=&B
 \prod_{\ell=1}^k  \frac{\lambda_{C_\ell}}{\mu_{\S(\{C_1,\ldots,C_\ell\})}}
 \left( \frac{\lambda_{C_1}+\cdots+\lambda_{C_\ell}}{\mu_{\S(\{C_1,\ldots,C_\ell\})}}\right)^{n_\ell} 
\end{eqnarray*} 

We also get the probability of a perfect match, $\emptyset$ from the balance equation:
\[
\pi_X(0) = \pi_X(0) \frac{\rmu}{\rlambda+\rmu} + \sum_{i=1}^I \pi_X(c_i)\frac{\mu_{\S(c_i)}}{\rlambda+\rmu}
\]
\[
\pi_X(0) \rlambda =  \sum_{i=1}^I \pi_X(c_i)\mu_{\S(c_i)} = 
B \sum_{i=1}^I \frac{\lambda_{c_i}}{\mu_{\S(c_i)}} \mu_{\S(c_i)} = B \rlambda,
\]
so we get $\pi_X(\emptyset) = B$, and from this also $\pi_Y(\emptyset) = B$.  
\qed  \end{proof}

We also obtain  the stationary probabilities of observing a state in which $(C_1,\ldots,C_k)$ appear for the first time  in that order, which we denote as $(C_1,\cdot,\ldots,C_k,\cdot)$: 
\begin{eqnarray}
\label{eqn.stationaryYperm}
\pi_Y(C_1,\cdot,\ldots,C_k,\cdot) &=& 
B \prod_{\ell=1}^k  \frac{\lambda_{C_\ell}}{\mu_{\S(\{C_1,\ldots,C_\ell\})}}
\left[ \sum_{n_\ell= 0}^\infty  \left( \frac{\lambda_{\{C_1,\ldots,C_\ell\}} }{ \mu_{\S(\{C_1,\ldots,C_\ell\})} } \right)^{n_\ell} \right]  \nonumber \\
&=&  B   \prod_{\ell=1}^k  \frac{\lambda_{C_\ell}}{\mu_{\S(\{C_1,\ldots,C_\ell\})}}
\frac{ \mu_{\S(\{C_1,\ldots,C_\ell\})} }{\mu_{\S(\{C_1,\ldots,C_\ell\})} - \lambda_{\{C_1,\ldots,C_\ell\}} } \\
&=&  B   \prod_{\ell=1}^k  
\frac{\lambda_{C_\ell} }{\mu_{\S(\{C_1,\ldots,C_\ell\})} - \lambda_{\{C_1,\ldots,C_\ell\}} }
\nonumber
\end{eqnarray}
Together with $\pi_Y(\emptyset) = B$, this gives us the stationary distribution of the 
unmatched types in $X(N)$ or in $U(N)$, in the order in which they appear for the first time.
These probabilities play a central role in the calculation of matching rates (Section \ref{sec.matchingrates}) and in deriving the distribution of delays (Section \ref{sec.delays}).

We now calculate the normalizing constant, $B=\pi_Y(0)$.
  We need to sum over all subsets of $\C$, and for each subset $C\subseteq \C$ we need to sum over all the permutations of $C$,  
(altogeter the summation has $f(I)=\sum_{k=0}^I  \frac{I!}{(I-k)!}$ terms, e.g. $f(0),f(1),\ldots,f(10)=1, 2, 5, 16, 65, 326, 1957, 13700, 109601, 986410, 9864101$,   we note that $n! < f(n) < (n+1)!$).
\begin{equation}
\label{eqn.constant}
B = \left(\sum_{k=0}^I  \quad \sum_{C\subseteq\C,\,|C|=k}\quad \sum_{(C_1,\ldots,C_k) \in \P(C)}  \quad \prod_{\ell=1}^k \quad
\frac{\lambda_{C_\ell}}{\mu_{\S(\{C_1,\ldots,C_\ell\})} - \lambda_{\{C_1,\ldots,C_\ell\}} }
  \right)^{-1}
\end{equation}
where the term for $k=0$ equals 1.
We strongly suspect that calculation of $B$ is $\sharp$-P complete (this is similar to the conjecture that the normalizing constant in symmetric  FCFS infinite bipartite matching is $\sharp$-P complete, made in \cite{adan-weiss:12}).

\subsection{Deriving the expression for $\pi_Y(C_1,\cdot,\ldots,C_k,\cdot)$ directly from Theorems \ref{thm.detailed},  \ref{thm.characterization}.}
\label{sec.sanity}
We first calculate $\pi_Y(C_1,\cdot)$, by summing $\pi_U(u)$ over all admissible $u$ that contain only unmatched agents of type $C_1$.  A typical state of this form will be:
\[
u = (C_1,\ldots,\tc,\tc,\ldots, s_j\in\overline{\S(C_1)},\ldots,\tc,\tc,\ldots, \ldots C_1 ,\ldots, \ldots 
C_1,\ldots,\tc,\tc,\ldots, s_j\in\overline{\S(C_1)},\ldots,\tc,\tc,\ldots)
\]
where $\tc$ can be any exchanged agents, $s_j$ must all be  goods incompatible with $C_1$, and 
there are $n\ge 1$ $C_1$s, with $m \ge 0$ $s_j\in\overline{\S(C_1)}$ following each $C_1$, and $k\ge 0$ $\;\tc$s following each $C_1$ or $s_j$.

We therefore get by the expression (\ref{eqn.detailedstationary}) for $\pi_U(u)$:
\begin{eqnarray*}
\pi_Y(C_1,\cdot) &=& B\sum_{n=1}^\infty \left\{  \frac{\rlambda}{\rlambda+\rmu} \alpha_{C_1} 
\left(\sum_{k=0}^\infty \big( \frac{\rlambda}{\rlambda+\rmu}\big)^k \right) 
\left[\sum_{m=0}^\infty  \left(\frac{\rmu}{\rlambda+\rmu} \beta_{\overline{\S(C_1)}} \sum_{k=0}^\infty \big( \frac{\rlambda}{\rlambda+\rmu}\big)^k \right)^m \right]
\right\}^n \\
&=& B\sum_{n=1}^\infty  \left\{ \frac{\rlambda}{\rlambda+\rmu}  \alpha_{C_1} 
 \frac{1}{1- \frac{\rlambda}{\rlambda+\rmu}}
 \sum_{m=0}^\infty \left(\frac{\rmu}{\rlambda+\rmu}  \beta_{\overline{\S(C_1)}} 
 \frac{1}{1- \frac{\rlambda}{\rlambda+\rmu}} \right)^m
\right\}^n \\
&=& B\sum_{n=1}^\infty \left\{ \frac{\rlambda}{\rmu} \alpha_{C_1} 
 \sum_{m=0}^\infty   \left( 1 - \beta_{\S(C_1)}\right)^m
\right\}^n \\
&=& B\sum_{n=1}^\infty \left\{ \frac{\rlambda}{\rmu}  \frac{\alpha_{C_1}}{\beta_{\S(C_1)}}
\right\}^n   =  
B\sum_{n=1}^\infty \left\{  \frac{\lambda_{C_1}}{\mu_{\S(C_1)}}  
\right\}^n = B \frac{\lambda_{C_1}}{\mu_{\S(C_1)} - \lambda_{C_1}} 
\end{eqnarray*}
Next we calculate $\pi_Y(C_1,\cdot,C_2,\cdot)$.   We need to sum probabilities of typical states  of the form:
\[
u= (C_1,\ldots,C_2, \ldots, s_j\in\overline{\S(C_1,C_2)},\ldots,c_i \in\{C_1,C_2\},\ldots,
\]
where the first part up to the first $C_2$ is as a typical $(C_1\cdot)$ state, followed by the first $C_2$ and then any number of $\tc$, and any number of $c_i \in\{C_1,C_2\}$ and $s_j\in\overline{\S(C_1,C_2)}$, and we then get:
\begin{eqnarray*}
\pi_Y(C_1,\cdot,C_2,\cdot) &=& B \sum_{n=1}^\infty \left\{  \frac{\rlambda}{\rlambda+\rmu} \alpha_{C_1} 
\left(\sum_{k=0}^\infty \big( \frac{\rlambda}{\rlambda+\rmu}\big)^k \right) 
\left[\sum_{m=0}^\infty  \left(\frac{\rmu}{\rlambda+\rmu} \beta_{\overline{\S(C_1)}} \sum_{k=0}^\infty \big( \frac{\rlambda}{\rlambda+\rmu}\big)^k \right)^m \right]
\right\}^n \\
&& \times \frac{\rlambda}{\rlambda+\rmu} \alpha_{C_2} 
\left(\sum_{k=0}^\infty \big( \frac{\rlambda}{\rlambda+\rmu}\big)^k \right) 
\left[\sum_{m=0}^\infty  \left(\frac{\rmu}{\rlambda+\rmu} \beta_{\overline{\S(\{C_1,C_2\})}} \sum_{k=0}^\infty \big( \frac{\rlambda}{\rlambda+\rmu}\big)^k \right)^m \right] \\
&& \times \sum_{n=0}^\infty \left\{  \frac{\rlambda}{\rlambda+\rmu} \alpha_{\{C_1,C_2\}} 
\left(\sum_{k=0}^\infty \big( \frac{\rlambda}{\rlambda+\rmu}\big)^k \right) 
\left[\sum_{m=0}^\infty  \left(\frac{\rmu}{\rlambda+\rmu} \beta_{\overline{\S(\{C_1,C_2\})}} \sum_{k=0}^\infty \big( \frac{\rlambda}{\rlambda+\rmu}\big)^k \right)^m \right]
\right\}^n \\
&=& B  \frac{\lambda_{C_1}}{\mu_{\S(C_1)} - \lambda_{C_1}} \times
 \frac{\lambda_{C_2}}{\mu_{{\S(\{C_1,C_2\})}}}  \times
 \sum_{n=0}^\infty \left\{ \frac{ \lambda_{\{C_1,C_2\}} }{\mu_{\S(\{C_1,C_2\})}}  \right\}^n \\
 &=& B  \frac{\lambda_{C_1}}{\mu_{\S(C_1)} - \lambda_{C_1}} 
 \frac{ \lambda_{C_2} }{\mu_{\S(\{C_1,C_2\})} -  \lambda_{\{C_1,C_2\}}}. 
\end{eqnarray*}
Finally we proceed by induction, to obtain, using the same steps in the derivation:
\begin{eqnarray*}
&&\pi_Y(C_1,\cdot,\ldots,C_L,\cdot) = B \prod_{l=1}^{L-1}  \frac{ \lambda_{C_l} }{\mu_{\S(\{C_1,\ldots,C_l\})} -  \lambda_{\{C_1,\ldots,C_l\}}}  \\
&&\quad \times \frac{\rlambda}{\rlambda+\rmu} \alpha_{C_L} 
\left(\sum_{k=0}^\infty \big( \frac{\rlambda}{\rlambda+\rmu}\big)^k \right) 
\left[\sum_{m=0}^\infty  \left(\frac{\rmu}{\rlambda+\rmu} \beta_{\overline{\S(\{C_1,\ldots,C_L\})}} \sum_{k=0}^\infty \big( \frac{\rlambda}{\rlambda+\rmu}\big)^k \right)^m \right] \\
&&\quad \times \sum_{n=0}^\infty \left\{  \frac{\rlambda}{\rlambda+\rmu} \alpha_{\{C_1,\ldots,C_L\}} 
\left(\sum_{k=0}^\infty \big( \frac{\rlambda}{\rlambda+\rmu}\big)^k \right) 
\left[\sum_{m=0}^\infty  \left(\frac{\rmu}{\rlambda+\rmu} \beta_{\overline{\S(\{C_1,\ldots,C_L\})}} \sum_{k=0}^\infty \big( \frac{\rlambda}{\rlambda+\rmu}\big)^k \right)^m \right]
\right\}^n \\
&& =B  \prod_{l=1}^{L}  \frac{ \lambda_{C_l} }{\mu_{\S(\{C_1,\ldots,C_l\})} -  \lambda_{\{C_1,\ldots,C_l\}}} 
\end{eqnarray*}

\section{Calculation of the Matching Rates}
\label{sec.matchingrates}
\begin{definition}[Matching Rates]
\label{def.matchingrates}
In a stable FCFS directed bipartite matching we define matching rates as follows:  $r_{s_j,c_i}$ is the fraction of matches which are of goods type $s_j$  and agent type $c_i$, and $r_{s_j,\emptyset}$ is the fraction of goods which are of type $s_j$ and are not matched.
\end{definition}
To make things clear, $\sum_{s_j\in\S} \left(r_{s_j,\emptyset} + \sum_{c_i\in\C} r_{s_j,c_i}\right) = 1$.  we also have:  
\begin{compactitem}[-]
\item
For goods of type $s_j$, the fraction that will remain unmatched, and the fraction that will be matched with agents of type $c_i$ is given by:
\[
\eta_{s_j}(\emptyset) = \frac{r_{s_j,\emptyset}}{ \left(r_{s_j,\emptyset} + \sum_{c_i\in\C} r_{s_j,c_i}\right)},
\qquad
\eta_{s_j}(c_i) = \frac{r_{s_j,c_i}}{ \left(r_{s_j,\emptyset} + \sum_{c_i\in\C} r_{s_j,c_i}\right)}.
\]
\item
For agents of type $c_i$, the fraction that will be matched with goods of type $s_j$ is given by:
\[
\theta_{c_i}(s_j) =  \frac{r_{s_j,c_i}}{\sum_{s_j\in\S} r_{s_j,c_i}}
\]
\end{compactitem}
The total fraction of unmatched goods is, by considering total rates of all  agents and goods:
\[
r_\emptyset = \sum_{s_j\in \S} \frac{\mu_{s_j}}{\rmu} \eta_{s_j,\emptyset} = \frac{\rmu - \rlambda}{\rmu}
\]

\begin{theorem}
In a stable system the matching rates for $(s_j,c_i)\in \G$ are given by:
\begin{eqnarray}
\label{eqn.matchingrates1}
&& r_{s_j,c_i} = \frac{\mu_{s_j}}{\rmu}  
\sum_{k=1}^I  \quad \sum_{C\subseteq\C,\,|C|=k}\quad \sum_{(C_1,\ldots,C_k) \in \P(C)} 
\pi_Y(C_1,\cdot,\ldots,C_k,\cdot) \\
&&\quad \times  \Big(1(C_1=c_i)\vee 1(\{C_1\}\cap\C(s_j)=\emptyset \wedge C_2=c_i) \vee 
\cdots 1(\{C_1,\ldots,C_{k-1}\}\cap\C(s_j)=\emptyset \wedge C_k=c_i)\Big) 
\nonumber
\end{eqnarray}
The rates of unmatched goods are given by:
\begin{eqnarray}
\label{eqn.matchingrates2}
&& r_{s_j,\emptyset} = \frac{\mu_{s_j}}{\rmu}  
\sum_{k=1}^I  \quad \sum_{C\subseteq\C,\,|C|=k}\quad \sum_{(C_1,\ldots,C_k) \in \P(C)} 
\pi_Y(C_1,\cdot,\ldots,C_k,\cdot) \\
&&\quad \times   \Big( 1(\{C_1,\ldots,C_k\}\cap\C(s_j)=\emptyset) 
\Big)
\end{eqnarray}
\end{theorem}
\begin{proof}
We consider the stationary version of $Y(N)$.  By discrete time version of PASTA, at the moment that $z_{N+1}\in \S$ the state at time $N$ is the stationary state of the process.  We wish to find for every $z_{N+1}\in \S$ what match it will find, or remain unmatched, and calculate the probability of each outcome.  These probabilities will be the matching rates.

Assume $z_{N+1} = s_j$, then the type of match for   
for $z_{N+1}$ is determined unequivocally by the state $C_1,\cdot,\ldots,C_k,\cdot$.   In particular the match will be with $c_i$ if and only if for some $1\le l \le k$, $C_l=c_i$, and all of $C_1,\ldots,C_{l-1}$ are not compatible with $s_j$.  If all of $C_1,\ldots,C_k$ are incompatible with $s_j$, then $z^{N+1}$ will remain unmatched.

Hence to obtain $r_{s_j,c_i}$ we need to add up all the stationary probabilities of states $(C_1,\cdot,\ldots,C_k,\cdot)$ for which  $s_j$ will match with $c_i$, and multiply this sum by  $\frac{\mu_{s_j}}{\rmu} $, which is  $P(z_{N+1}=s_j\,|\,z_{N+1}\in\S)$.  Similarly for $r_{s_j,\emptyset}$ we add the probabilities for $(C_1,\cdot,\ldots,C_k,\cdot)$ for which  $s_j$ remains unmatched and multiply by $\frac{\mu_{s_j}}{\rmu} $.
\qed  \end{proof}

As already stated in several previous papers, it is very likely that the calculation of $B$ is $\sharp$-P complete.  
Calculation of $B$ directly from (\ref{eqn.constant}) requires adding up $\sum_j j! {n\choose j}$ terms, which is still quite feasible for small $|\C|=I$.
We note that the terms added to calculate $B$ are also the terms needed to add for calculating the matching rates.  We formulate the following algorithm for calculation of the matching rates:

\begin{framed}
\begin{center}
algorithm 1:  Calculation of normalizing constant and matching rates
\end{center}
\begin{compactenum}
\item
Initialize:  $B:=1$; $r_{s_j,c_i}:=0$; 
\item
For $A \subseteq \C$, $A\ne \emptyset$:

For $(C_1,\ldots,C_{|A|}) \in \mbox{Permutations}(A)$:
\item  
\hspace{0.1in} $X=1$; $C_{set}:=\emptyset$; $S_{set}:=\emptyset$; $\lambda_{set}:=0$; 
$\mu_{set}:=0$;
   $\; \mbox{Match}(s_j) := \emptyset$.
\item
\hspace{0.1in} For $l=1,|A|$:
\item
\hspace{0.2in} $C_{set}:= C_{set} \cup C_l$;  
$\Delta S_{set}:=\S(C_l)\setminus S_{set}$;
$S_{set}:=S_{set} \cup \S(C_l)$;

\hspace{0.2in}  $\lambda=\lambda_{C_l}$;  $\lambda_{set}:=\lambda_{set}+\lambda$;
$\mu_{set}:=\mu_{set}+\mu_{\Delta S_{set}}$;  

\hspace{0.2in}  $R_l:= \frac{1}{ \mu_{set} - \lambda_{set} }$;  $X:=X \lambda R_l$.

\item
\hspace{0.1in} $B:=B+X$.
\item
\hspace{0.1in} For $l=1,|A|$:

\hspace{0.1in} For $s_j\in \S$:
\item
\hspace{0.2in}  If $C_l \in \C(s_j) \wedge \mbox{Match}(s_j)=\emptyset$ then:

\hspace{0.3in}  $ \mbox{Match}(s_j):=C_l$; 
$\; r_{s_j,C_l}:= r_{s_j,C_l} + X$.
\item
$B:=1/B$;
$r_{s_j,c_i}:=B\; \frac{\mu_{s_j}}{\rmu}\; r_{s_j,c_i} $;  
$r_{s_j,\emptyset}:= \frac{\mu_{s_j}}{\rmu} - \sum_{c_i\in\C} r_{s_j,c_i} $.
\end{compactenum}
\end{framed}
We explain the algorithm here:  Line 1 is initialization.  Lines 2--8 are  a loop over all subsets of agent types, and permutations of the subsets.  Within this loop, for each $C_1,\ldots,C_k$, there are two internal loops: lines 3--5 contain a loop over $C_l$, for the calculation of the  ratios,  
$\frac{\lambda_{C_l}}{ \mu_{\S(C_1,\ldots,C_l)} - \lambda_{C_1,\ldots,C_l} }$,   
and their successive products to obtain
$X \propto \pi_Y(C_1,\cdot,\ldots,C_k,\cdot)$,  then line 6 accumulates the calculated values of $X$;  the second loop, lines 7--8, is over all $C_l$ and $s_j\in \S$, it uses the value $X$, and accumulates it to $r_{s_j,c_i}$ if $Y(N)=(C_1,\cdot,\ldots,C_k,\cdot)$ and $Z^{N+1}=s_j$ results in an $(s_j,c_l)$ match.  Line 9 completes the calculation of the normalizing constant and matching rates.


\section{The Distribution and the Moments of  Delays}
\label{sec.delays}
\begin{definition}[Delays]
\label{def.delays}
If $z^m=c_i$ is matched to $z^n=s_j$, the we define the $(s_j,c_i)$ delay between them as $L_{s_j,c_i}=n-m$.    If $z^m=c_i$ is matched to $z^n \in \S$ we define the delay of $c_i$ by $L_{c_i}=n-m$.
\end{definition}
In this section we obtain the distribution of the delays.  It turns out that this distribution is a mixture of convolutions of geometric random variables.  We will derive expressions for its generating function, and   based on the simple form of the distribution we  also  obtain expressions for its moments.  
We will also translate the delays to waiting times, if the arrivals of all types form independent Poisson processes.
To calculate the delay we need the following result that complements Theorems \ref{thm.detailed}, \ref{thm.characterization}.

\begin{theorem}
\label{thm.geometric}
Consider the process $U(N)$, and assume that the corresponding $Y(N)=(C_1,\cdot,\ldots,C_k,\cdot)$.  Dente by $D_l(C_1,\ldots,C_k)$  the number of items in $U(N)$,  From the first appearance of $C_l$ to the first appearance of $C_{l+1}$, $l=1,\ldots,k-1$ and from the first appearance of $C_k$ to the end, i.e. if $U(N)=u=(u_1,\ldots,u_L)$, and  $u_{m_l}=C_l,\,u_{m_{l+1}}=C_{l+1}$ are the first appearances of agent types $C_l,\,C_{l+1}$ in $u$, then $D_l(C_1,\ldots,C_k)=m_{l+1} - m_l, \, l=1,\ldots,k-1$, and 
$D_k(C_1,\ldots,C_k)=L+1 - m_k$.  The distribution of $D_l(C_1,\ldots,C_k)$ is a Geometric distribution over $m=1,2,\ldots$, given by:
\begin{equation}
\label{eqn.interval}
P(D_l(C_1,\ldots,C_k)=m )= \frac{\mu_{\S(\{C_1,\ldots,C_l\})} -  \lambda_{\{C_1,\ldots,C_l\}}}{\rlambda+\rmu}
\left( 1- \frac{\mu_{\S(\{C_1,\ldots,C_l\})} -  \lambda_{\{C_1,\ldots,C_l\}}}{\rlambda+\rmu} \right)^{m-1}.
\end{equation}
Furthermore, $D_l(C_1,\ldots,C_k),\,l=1,\ldots,k$ are independent.  
\end{theorem}
\begin{proof}
Note that by the condition (\ref{eqn.stability}), for a stable system, $0 < \frac{\mu_{\S(\{C_1,\ldots,C_l\})} -  \lambda_{\{C_1,\ldots,C_l\}}}{\rlambda+\rmu} < 1$.
We fix $C_1,\ldots,C_k$, and for notational convenience use $D_l$ to denote $D_l(C_1,\ldots,C_k)$.
We   calculate the probability of the states of $U(N)$, with  \newline
$(\overbrace{C_1,\cdots}^{m_1},\overbrace{C_2,\cdots}^{m_2},\ldots,\overbrace{C_k,\cdots}^{m_k})$
where we sum up over all possible values of each of the  $m_l-1$ items  in the interval between  the first appearance of $C_l$ and the first appearance of $C_{l+1}$, i.e. we will calculate:
$P(D_1=m_1,\ldots,D_k=m_k \cap Y=C_1,\cdot,\ldots,C_k,\cdot)$. 

To do that we calculate first:
\begin{eqnarray*}
&&  \frac{P(D_1=m_1+1\cap Y=C_1,\cdot)}{P(D_1=m_1\cap Y=C_1,\cdot)} 
 =\frac{P(\overbrace{C_1,\cdots}^{m_1+1})}{P(\overbrace{C_1,\cdots}^{m_1})}
  = \sum_{u\in \{C_1 ,\overline{\S(C_1)},\tc_i,i=1,\ldots,I \}}
\frac{P(\overbrace{C_1,\cdots}^{m_1},u)}{P(\overbrace{C_1,\cdots}^{m_1})} \\ 
&&=\frac{1}{\rlambda+\rmu} \Big( \lambda_{C_1} + \mu_{\overline{\S(C_1)}}+\rlambda \Big)
=  1 - \frac{\mu_{\S(C_1)}-\lambda_{C_1}}{\rlambda+\rmu}
\end{eqnarray*}
From which we obtain:
\[
P(D_1=m_1\cap Y=C_1,\cdot) = \frac{\lambda_{C_1}}{\rlambda+\rmu} 
\left( 1 - \frac{\mu_{\S(C_1)}-\lambda_{C_1}}{\rlambda+\rmu} \right)^{m_1-1}.
\]
Proceeding in the same way for $D_2,\ldots,D_k$ we obtain that 
\[
P(D_1=m_1,\ldots,D_k=m_k \cap C_1,\cdot,\ldots,C_k,\cdot) =  \prod_{l=1}^k
\frac{\lambda_{C_l}}{\rlambda+\rmu} 
\left( 1 - \frac{ \mu_{\S(C_1,\ldots,C_l)}-\lambda_{C_1,\ldots,C_l} }{\rlambda+\rmu} \right)^{m_l-1}
\]
Dividing by $\pi_Y(C_1,\cdot,\ldots,C_k,\cdot)$ given in equation (\ref{eqn.stationaryYperm}) we obtain
\begin{eqnarray*}
&& P\big(D_1=m_1,\ldots,D_k=m_k \big| Y(N)= (C_1,\cdot,\ldots,C_k,\cdot) \big) \\
&& \quad =   \prod_{l=1}^k
\frac{ \mu_{\S(C_1,\ldots,C_l)}-\lambda_{C_1,\ldots,C_l}}{\rlambda+\rmu} 
\left( 1 - \frac{ \mu_{\S(C_1,\ldots,C_l)}-\lambda_{C_1,\ldots,C_l} }{\rlambda+\rmu} \right)^{m_l-1}
\end{eqnarray*}
Which is what we needed to show.
\qed\end{proof}

By definition, if $Y(N)=(C_1,\cdot,\ldots,C_k,\cdot)$, with $C_l=c_i$, and $z^{N+1}=s_j$, so that  $s_j$ matches with $C_l$ (i.e. $s_j\not\in \S(C_1,\ldots,C_{l-1})$ and $s_j\in\S(C_l)$),  then 
\[
L_{s_j,c_i} = \sum_{h=l}^k  D_h(C_1,\ldots,C_k)
\]
and this will happen with probability $ \pi_Y((C_1,\cdot,\ldots,C_k,\cdot) \frac{\mu_{s_j}}{\rmu}$.  
We obtain directly:
\begin{theorem}
The delays  $L_{s_j,c_i}$ and of $L_{c_i}$ are distributed as mixtures of convolutions of geometric random variables: 
\begin{eqnarray} 
\label{eqn.delays1}
&& L_{s_j,c_i} =  \frac{1}{r_{s_j,c_i}}\frac{\mu_{s_j}}{\rmu} 
\sum_{k=1}^I  \quad \sum_{C\subseteq\C,\,|C|=k}\quad \sum_{(C_1,\ldots,C_k) \in \P(C)} 
\sum_{l=1}^k 
\pi_Y(C_1,\cdot,\ldots,C_k,\cdot)  \nonumber \\
&&\quad  \times  1\big( (C_l=c_i)\wedge (s_j \not\in \S(C_1,\ldots,C_{l-1}))\big) 
\times  \sum_{h=l}^k  D_h(C_1,\ldots,C_k)
\end{eqnarray}
\begin{equation}
\label{eqn.delays2}
 L_{c_i} =  \sum_{s_j\in\S(c_i)} \frac{ r_{s_j,c_i} }{  \sum_{s_h\in\S(c_i)}  r_{s_h,c_i} }
L_{s_j,c_i}(\sfZ).
\end{equation}
\end{theorem}
\begin{proof}
For $L_{s_j,c_i}$ we add up the values conditional of $Y=(C_1,\cdot,\ldots,C_k,\cdot)$, for all $Y$ for which there is an $(s_j,c_i)$ match, multiplied by the probability $\pi_Y(C_1,\cdot,\ldots,C_k,\cdot)$ and the probability of $s_j$ which is $\frac{\mu_{s_j}}{\rmu}$,   and divide by the sum of these probabilities which is $r_{s_j,c_i}$.   $L_{c_i}$ is an average of the $L_{s_j,c_i}$, weighted by $r_{s_j,c_i}$.
\qed \end{proof}

\begin{corollary}
The generating function of $L_{s_j,c_i}$ is given by:
\begin{eqnarray}
\label{eqn.delaysgen}
&& G_{s_j,c_i}(\sfZ) =  \frac{1}{r_{s_j,c_i}}\frac{\mu_{s_j}}{\rmu} 
\sum_{k=1}^I  \quad \sum_{C\subseteq\C,\,|C|=k}\quad \sum_{(C_1,\ldots,C_k) \in \P(C)} 
\sum_{l=1}^k 
\pi_Y(C_1,\cdot,\ldots,C_k,\cdot)   \\
&&\quad  \times  1\big( (C_l=c_i)\wedge (s_j \not\in \S(C_1,\ldots,C_{l-1}))\big) 
\times  \prod_{h=l}^k 
\frac{ \sfZ \frac{ \mu_{\S(C_1,\ldots,C_l)}-\lambda_{C_1,\ldots,C_l}}{\rlambda+\rmu} }
{ 1 - \sfZ\left( 1 - \frac{ \mu_{\S(C_1,\ldots,C_l)}-\lambda_{C_1,\ldots,C_l} }{\rlambda+\rmu} \right) }, 
\nonumber
\end{eqnarray}
with mean and variance:
\begin{eqnarray}
\label{eqn.delaysmoment}
&& E(L_{s_j,c_i}) =  \frac{1}{r_{s_j,c_i}}\frac{\mu_{s_j}}{\rmu} 
\sum_{k=1}^I  \quad \sum_{C\subseteq\C,\,|C|=k}\quad \sum_{(C_1,\ldots,C_k) \in \P(C)} 
\sum_{l=1}^k 
\pi_Y(C_1,\cdot,\ldots,C_k,\cdot)   \\
&&\quad  \times  1\big( (C_l=c_i)\wedge (s_j \not\in \S(C_1,\ldots,C_{l-1}))\big) 
\times  \sum_{h=l}^k 
\left(  \frac{ \mu_{\S(C_1,\ldots,C_l)}-\lambda_{C_1,\ldots,C_l}}{\rlambda+\rmu} \right)^{-1},
\nonumber  \\
\nonumber  \\
&& Var({L_{s_j,c_i}}) = \big\{ \frac{1}{r_{s_j,c_i}}\frac{\mu_{s_j}}{\rmu} 
\sum_{k=1}^I  \quad \sum_{C\subseteq\C,\,|C|=k}\quad \sum_{(C_1,\ldots,C_k) \in \P(C)} 
\sum_{l=1}^k 
\pi_Y(C_1,\cdot,\ldots,C_k,\cdot)  \nonumber \\
&&\quad  \times  1\big( (C_l=c_i)\wedge (s_j \not\in \S(C_1,\ldots,C_{l-1}))\big)  \\
&&\quad  \times \left[ \sum_{h=l}^k 
\frac{ \left(1 -  \frac{ \mu_{\S(C_1,\ldots,C_l)}-\lambda_{C_1,\ldots,C_l}}{\rlambda+\rmu} \right) }
{  \left(  \frac{ \mu_{\S(C_1,\ldots,C_l)}-\lambda_{C_1,\ldots,C_l}}{\rlambda+\rmu} \right)^2 }
+ \left( \sum_{h=l}^k 
\left(  \frac{ \mu_{\S(C_1,\ldots,C_l)}-\lambda_{C_1,\ldots,C_l}}{\rlambda+\rmu} \right)^{-1}
\right)^2 \right]  \Big\}
\nonumber\\
&& \qquad  -  E(L_{s_j,c_i}) ^2
\nonumber  
\end{eqnarray}
The mean and variance of $L_{c_i}$ can be obtained from these directly.
\end{corollary}
\begin{proof}
this follows from properties of Geometric random variables:
\begin{eqnarray*}
&& X\sim Geometric(p), \qquad  P(X=n)=p (1-p)^{n-1}, \quad n=1,2,\ldots, \\
&&  G_X(\sfZ) = E(\sfZ^X)=\frac{\sfZ p}{1- \sfZ(1-p)},  \qquad E(X)= p^{-1},  \qquad
Var(X) = \frac{1-p}{p^2}.
\end{eqnarray*}
Conditional on $Y=(C_1,\ldots,C_k)$, the delay $L_{s_j,c_i}$ is a sum of independent geometric random variables, with mean and variance that are the sums of the means and variances of the geometric random variables.   The unconditional mean follows directly.  The unconditional variance is the expectation of the conditional variances plus the variance of the conditional means.  
\qed \end{proof}

Calculation of means and variances of the delays is similar to the calculation of the matching raters.  We provide here an algorithm for these calculations.  It follows the steps of Algorithm 1, and only requires addition of some calculations to line 8, and additional line 10.
\begin{framed}
\begin{center}
algorithm 2:  Calculation of means and variances of delays
\end{center}
\begin{compactenum}
\item[1a]
Initialize: $XE_{s_j,c_i}:=0$; $\quad XE^2_{s_j,c_i}:=0$; $\quad XV_{s_j,c_i}:=0$.

$\quad \vdots$
\item[8a]
\hspace{0.2in}  If $C_l \in \C(s_j) \wedge \mbox{Match}(s_j)=\emptyset$ then:

\hspace{0.3in}  $ \mbox{Match}(s_j):=C_l$; 
$\; r_{s_j,C_l}:= r_{s_j,C_l} + X$.

\hspace{0.3in}  $AE:=0$;   $\quad AV:=0$.
\item[8b]
\hspace{0.3in}  For $h=l,|A|$:

\hspace{0.4in}  $P_h=1/[(\rlambda + \rmu)R_h]$; 

\hspace{0.4in}  $AE:=AE + 1/P_h$,    $\quad  AV:=AV +  (1 - P_h)/{P_h}^2$.
\item[8c]
\hspace{0.2in} $E_{s_j,C_l}: = E_{s_j,C_l} + X\,AE$;  $ E^2_{s_j,C_l}: = E^2_{s_j,C_l} + X\,AE^2$;  
$ V_{s_j,C_l}: = V_{s_j,C_l} + X\,AV$.

$\quad \vdots$
\item[9a]
$E(L_{s_j,c_i}) := B \frac{1}{r_{s_j,c_i}} \frac{\mu_{s_j}}{\rmu}  E_{s_j,c_i}$;
$E^2_{s_j,C_i} := B \frac{1}{r_{s_j,c_i}} \frac{\mu_{s_j}}{\rmu}  E^2_{s_j,C_i}$;
$V_{s_j,C_i} := B \frac{1}{r_{s_j,c_i}} \frac{\mu_{s_j}}{\rmu}  V_{s_j,C_i}$.

$Var(L_{s_j,c_i}):=V_{s_j,C_l} + E^2_{s_j,C_i} - E(L_{s_j,c_i})^2$.
\end{compactenum}
\end{framed}
We explain the algorithm:  inside the loop of line 8, for  $(C_1,\ldots,C_k)$ we locate matches $(s_j,C_l)$ in the {\bf If} statement, and then calculate the sum from $l$ to $k$ of the conditional mean and variance of $D_l(C_1,\ldots,C_k)$, multiplied by  $X \propto \pi_Y(C_1,\ldots,C_k)$.   We accumulate these over all  $(C_1,\ldots,C_k)$,  to calculate the unconditional mean of the conditional variances, and the unconditional variance of the conditional means, leading to  the means and variances in line 9a.

Delays can be translated to waiting times.
Assume arrivals of agents and goods are independent Poisson processes, with rates $\lambda_{c_i}$, $\mu_{s_j}$,
if $z^m=c_i$ is matched to $z^n=s_j$, then we denote the $(s_j,c_i)$ wait between them as $W_{s_j,c_i}$, which is the sum of $n-m$ exponential random variables, with rate parameter $\rlambda_\rmu$.  If $z^m=c_i$ is matched to $z^n \in \S$ we denote the wait of $z^m$ by $W_{c_i}$, ditto.
We then have:
\begin{corollary}[Waiting Times]
\label{thm.waits}
Under Poisson assumption the waiting times are distributed as mixtures of convolutions of independent exponential random variables. 
The moment generating function  of $W_{s_j,c_i}$ is given by:
\begin{eqnarray}
\label{eqn.waitgen}
&& M_{s_j,c_i}(\sfs) =  \frac{1}{r_{s_j,c_i}}\frac{\mu_{s_j}}{\rmu} 
\sum_{k=1}^I  \quad \sum_{C\subseteq\C,\,|C|=k}\quad \sum_{(C_1,\ldots,C_k) \in \P(C)} 
\sum_{l=1}^k 
\pi_Y(C_1,\cdot,\ldots,C_k,\cdot)   \\
&&\quad  \times  1\big( (C_l=c_i)\wedge (s_j \not\in \S(C_1,\ldots,C_{l-1}))\big) 
\times  \prod_{h=l}^k 
\frac{ \mu_{\S(C_1,\ldots,C_l)}-\lambda_{C_1,\ldots,C_l}}
{ \mu_{\S(C_1,\ldots,C_l)}-\lambda_{C_1,\ldots,C_l} -  \sfs}
\nonumber
\end{eqnarray}
with mean and  variance:
\begin{eqnarray}
\label{eqn.waitmean}
&& E(W_{s_j,c_i}) =  \frac{1}{r_{s_j,c_i}}\frac{\mu_{s_j}}{\rmu} 
\sum_{k=1}^I  \quad \sum_{C\subseteq\C,\,|C|=k}\quad \sum_{(C_1,\ldots,C_k) \in \P(C)} 
\sum_{l=1}^k 
\pi_Y(C_1,\cdot,\ldots,C_k,\cdot)   \\
&&\quad  \times  1\big( (C_l=c_i)\wedge (s_j \not\in \S(C_1,\ldots,C_{l-1}))\big) 
\times  \sum_{h=l}^k 
  \frac{1}{ \mu_{\S(C_1,\ldots,C_l)}-\lambda_{C_1,\ldots,C_l}},
\nonumber  
\end{eqnarray}
\begin{eqnarray}
\label{eqn.waitvar}
&& Var({W_{s_j,c_i}}) = \big\{ \frac{1}{r_{s_j,c_i}}\frac{\mu_{s_j}}{\rmu} 
\sum_{k=1}^I  \quad \sum_{C\subseteq\C,\,|C|=k}\quad \sum_{(C_1,\ldots,C_k) \in \P(C)} 
\sum_{l=1}^k 
\pi_Y(C_1,\cdot,\ldots,C_k,\cdot)  \nonumber \\
&&\quad  \times  1\big( (C_l=c_i)\wedge (s_j \not\in \S(C_1,\ldots,C_{l-1}))\big)  \\
&&\quad  \times \left[ \sum_{h=l}^k 
\frac{1}{(\mu_{\S(C_1,\ldots,C_l)}-\lambda_{C_1,\ldots,C_l})^2}
+ \left( \sum_{h=l}^k \frac{1}{\mu_{\S(C_1,\ldots,C_l)}-\lambda_{C_1,\ldots,C_l}} \right)^2
\right]  \Big\}
-  E(W_{s_j,c_i}) ^2
\nonumber  
\end{eqnarray}
The mean and variance of $W_{c_i}$ can be obtained from these directly.
\end{corollary}
\begin{proof}
Recall that if $U\sim Exp(\theta)$ then its moment generating function is $M_U(\sfs) = \frac{\theta}{\theta-\sfs}$, it mean is $E(U)=\frac{1}{\theta}$ and its variance is $Var(U)=\frac{1}{\theta^2}$.

If $U_i \sim Exp(\theta)$ i.i.d., and $Y\sim Geom(p)$ then $V=\sum_{i=1}^Y U_i \sim Exp(\theta p)$.
Hence, each $D_l(C_1,\ldots,C_k)$ contributes an exponential wait  with rate parameter 
$ \mu_{\S(C_1,\ldots,C_k)} -  \lambda_{C_1,\ldots,C_k}$.
\qed \end{proof}


\section{Light and Heavy Traffic, Relation to Symmetric FCFS Matching}
\label{sec.limits}
In this section we consider our system when the overall traffic intensity, given by $\rlambda/\rmu$ is close to 0, i.e. the system is in light traffic,  or close to and below 1, i.e. the system is in heavy traffic.  We wish to obtain the limiting matching rates under these conditions.

For light traffic we have:
\begin{proposition}
\label{thm.lighttraffic}
If $\rlambda/\rmu \to 0$, then for $(s_j,c_i)\in \G$,  $r_{s_j,c_i} \to \alpha_{c_i} \frac{\mu_{s_j}}{ \sum_{s_k \in \S(c_i)} \mu_{s_k}}$
\end{proposition}
\begin{proof}
If $\rlambda/\rmu \to 0$ then for  agent  $z^N=c_i$, with high probability there will be no earlier unmatched agents prior to $Z^N$, i.e. $X(N)=\emptyset$, and $z^N$  will  be followed 
 by a large number of goods of all types.  $z^N$ will then match with the first good with which it is compatible,  and this good will be of type $s_j$ with probability 
 $\frac{\mu_{s_j}}{ \sum_{s_k \in \S(c_i) } \mu_{s_k}}$.
 To obtain $r_{s_j,c_i}$ this needs to be multiplied by $\alpha_{c_i}$.
\qed \end{proof}

We next define a property which ensures stability for all    $0<\rlambda/\rmu < 1$.
\begin{definition}[Complete Resource Pooling]
\label{def.crp}
For given compatibility graph $\G$ and frequency vectors $\alpha$, $\beta$,
we say that the {\em system has complete resource pooling} (CRP)  if the following three equivalent conditions hold for all subsets $C\subset \C$, $C\ne \emptyset, \C$, and $S\subset \S$, $S\ne \emptyset,\S$:
\begin{equation}
\alpha_C < \beta_{\S(C)},  \qquad   \beta_S < \alpha_{\C(S)},  \qquad  \alpha_{\U(S)} < \beta_S.
\end{equation}
\end{definition}
This condition is  also the necessary and sufficient condition for stability of the symmetric FCFS bipartite matching model of \cite{caldentey-kaplan-weiss:09,adan-weiss:12,adan-busic-mairesse-weiss:17}.
By Theorem \ref{thm.uniqueness}, CRP implies that our system is stable for all  $0<\rlambda/\rmu < 1$.
If CRP does not hold, then the system will be stable if and only if
\[
0<\rlambda/\rmu <  \min_{C\ne\emptyset, \C} \frac{ \beta_{\S(C)}}{\alpha_C }.
\]

As we shall see in the next theorem, in heavy traffic our model behaves more and more like the symmetric FCFS bipartite matching model of \cite{caldentey-kaplan-weiss:09,adan-weiss:12,adan-busic-mairesse-weiss:17}.  For clarity we will denote quantities relating to our directed FCFS matching model by $\cdot^D$, and quantities relating to the symmetric FCFS bipartite matching model by $\cdot^B$

It was noted in Moyal, Busic and Mairesse \cite{moyal-busic-mairesse:17} that in a FCFS matching model with a single sequence of arrivals of both agents and goods, if agents as well as goods are allowed to wait until they find a match the system can never be stable, even if the data satisfies the property of CRP, since the total number of agents minus the total number of goods in the sequence starting from empty, behaves like an unconstrained one dimensional random walk, which is always either transient if CRP does not hold, or null recurrent if CRP holds.  Therefore  the number of unmatched agents or goods diverges in the sense that if we define: \newline
$\Delta(N) =$ absolute difference between number of unmatched  agents and unmatched goods by time $N$,  then  $P(\Delta(N)>n) \to 1$ as $N\to\infty$.    

However, in the symmetric FCFS matching model of \cite{caldentey-kaplan-weiss:09,adan-weiss:12,adan-busic-mairesse-weiss:17} there are two sequences, one only of agents and one only of goods. Starting from empty, by time $N$ the number of agent arrivals equals the number of goods arrivals, and if CRP holds then the system is stable,  and the number of unmatched agents or goods converges to a stationary distribution. 

For our system we defined the Markov processes $X^D(N)=(c^1,c^2,\ldots,c^L)$ that lists all unmatched agents at time $N$ (Section \ref{sec.previous}).
For the symmetric FCFS matching model in \cite{adan-busic-mairesse-weiss:17} we define the Markov process $X^B(N)=(c^1,c^2,\ldots,c^L)$ which list unmatched agents after all goods up to $N$ have been matched to agents earlier or later than $N$, and we include in this list all the unmatched agents starting from the first unmatched agent, and continuing until at least one type of agent is included in the list.  
This process is a slight modification of the so called 'natural' process defined and studied in Section 5.5 of 
\cite{adan-busic-mairesse-weiss:17}).  
The stationary distribution of $X^D(N)$ was given in (\ref{eqn.stationaryY}).  The stationary distribution of $X^B(N)$ can be derived in the same way that the stationary distribution of the  'natural' process in \cite{adan-busic-mairesse-weiss:17} is derived.  These stationary distributions are:
\begin{equation}
\label{eqn.compare}
\begin{array}{l}
\pi_X^D(c^1,\ldots,c^L)  \propto  \left(\frac{\rlambda}{\rmu}\right)^L
\prod_{\ell=1}^L \frac{\alpha_{c^\ell}}{\beta_{\S(\{c^1,\ldots,c^\ell\})}} \\
\pi_X^D(c^1,\ldots,c^L)  \propto 
\prod_{\ell=1}^L \frac{\alpha_{c^\ell}}{\beta_{\S(\{c^1,\ldots,c^\ell\})}} 
\end{array}
\end{equation}
We also defined the Markov process $Y^D(N)=(C_1,n_1,\ldots,\C_k,n_k)$ (Section \ref{sec.constant}).   
We define similarly the process $Y^B(N)=(C_1,n_1,\ldots,C_{I-1},n_{I-1},C_I)$ which is obtained from 
$X^B(N)$ by listing the types of agents in their order of appearance, and counting the number of unmatched agents between each pair.  Note that $C_I$ is determined as the one type not included in $C_1,\ldots,C_{I-1}$, and it is followed by the infinite sequence of agents of all types not yet matched.
The stationary distribution of  $Y^B(N)$ is derived similar to the derivation of the stationary distribution of  $Y^D(N)$ in Theorem  \ref{thm.succint}.  We have:
\begin{equation}
\label{eqn.compare2}
\begin{array}{l}
\pi_Y^D(C_1,n_1,\ldots,C_k,n_k)  \propto  
\prod_{\ell=1}^k  \left(\frac{\rlambda}{\rmu}\frac{\alpha_{C_\ell}}{\beta_{\S(\{C_1,\ldots,C_\ell\})}}\right)
\left(\frac{\rlambda}{\rmu} \frac{\alpha_{\{C_1,\ldots,C_\ell\}}}{\beta_{\S(\{C_1,\ldots,C_\ell\})}} \right)^{n_\ell}\\
\pi_Y^D(C_1,n_1,\ldots,C_{I-1},n_{I-1},C_I)  \propto   \alpha_{C_I}
\prod_{\ell=1}^{I-1}  \frac{\alpha_{C_\ell}}{\beta_{\S(\{C_1,\ldots,C_\ell\})}}
\left( \frac{\alpha_{\{C_1,\ldots,C_\ell\}}}{\beta_{\S(\{C_1,\ldots,C_\ell\})}} \right)^{n_\ell}
\end{array}
\end{equation}

\begin{theorem}
\label{thm.heavytraffic}
Assume CRP holds.
When $\frac{\rlambda}{\rmu} \nearrow 1$ the following holds for the stationary distribution of $Y^D(N)$:
\begin{compactenum}[(i)]
\item
The probability that in $Y^D(N)=(C_1,n_1,\ldots,C_k,n_k)$ with $k<I$ converges to 0.
\item
The probability that $Y^D(N)=(C_1,n_1,\ldots,C_I,n_I)$ and $n_I$ is less than any constant goes to 0.
\item
Conditional on $k=I, n_I>0$, the process $Y_0^D(N)=(C_1,n_1,\ldots,C_I)$, obtained from $Y^D(N)$ by deleting $n_I$ from the state, is a Markov process.
\item
The stationary distribution and the transition probabilities of  $Y_0^D(N)$ converge to the stationary distribution and  the transition probabilities of  $Y^B(N)$ in total variation distance.
\end{compactenum}
\end{theorem}
\begin{proof}
(i)  We note that $ \frac{\alpha_{\{C_1,\ldots,C_I\}}}{\beta_{\S(\{C_1,\ldots,C_I\})}} = 1$.   By Theorem \ref{thm.geometric} all the $n_i$ have geometric distributions.  For $l<I$ we have by CRP
\[
\frac{\rlambda}{\rmu} \frac{\alpha_{\{C_1,\ldots,C_\ell\}}}{\beta_{\S(\{C_1,\ldots,C_\ell\})}} <
\frac{\alpha_{\{C_1,\ldots,C_\ell\}}}{\beta_{\S(\{C_1,\ldots,C_\ell\})}} 
\le \max_{C\subset \C} \frac{\alpha_C}{\beta_{\S(C)}} = \kappa < 1.
\]
while for $k=I$ we have $\frac{\rlambda}{\rmu}  \frac{\rlambda}{\rmu} \frac{\alpha_{\{C_1,\ldots,C_\ell\}}}{\beta_{\S(\{C_1,\ldots,C_\ell\})}} =\frac{\rlambda}{\rmu} \nearrow 1$.  It follows that most of the sum of terms needed to compute the normalizing constant $B$ is concentrated in the terms for which $k=I$.

(ii)  When $k=I$ then $n_I$ is distributed as a geometric random variable  and so the probability that $n_I$ is less than any constant $M$ 
is $1 - \left(\frac{\rlambda}{\rmu}\right)^M \searrow 0$.

(iii)  If $k=I$ and $n_i> >0$, the possible transitions from $Y^D(N)$ are: if $z^{N+1}=c_i$ it augments $n_i$.  Otherwise,  if $z^{N+1}=s_j$  it will match with one of $C_1,\ldots,C_I$, say $C_l$, and as a result 
the order of $C_1,\ldots,C_I$ and number of unmatched agents between them will change.  At worst, $C_l$ appears only once, in which case it become the new $C_I$ and $n_I$ decreases by a geometric number of terms.  We do not need the information on $n_I$ except to know it is positive enough.

(iv)  It is seen immediately that $\pi_{Y^D_0}(C_1,n_1,\ldots,C_I) \to \pi_{Y^B}(C_1,n_1,\ldots,C_I)$.  The description of the transitions in part (iii) shows also that the conditional transitions of $\pi_{Y^D_0}$ when $n_I>>0$ are the same as those of $Y^B$.
\qed \end{proof}

It follows immediately that:
\begin{corollary}
Assume CRP holds.
When $\frac{\rlambda}{\rmu} \nearrow 1$ the  matching rates of the directed FCFS bipartite matching model converge to the matching rates of the symmetric FCFS bipartite matching model, while the fraction of unassigned goods converges to 0.
\end{corollary}



\section{A Simple Illustrative Example}
\label{sec.example}
We consider a simple system, with three types of agents and three types of goods and $\G=\{(s_1,c_1),(s_1,c_2),(s_2,c_1),(s_2,c_3),(s_3,c_2),(s_3,c_3)\}$, with $\alpha=(.3,.5,.2)$, $\beta=(.3,.3,.4)$, as illustrated in Fig \ref{fig.3x3system}.
\begin{figure}[htbp]
   \centering
   \includegraphics[scale=0.5]{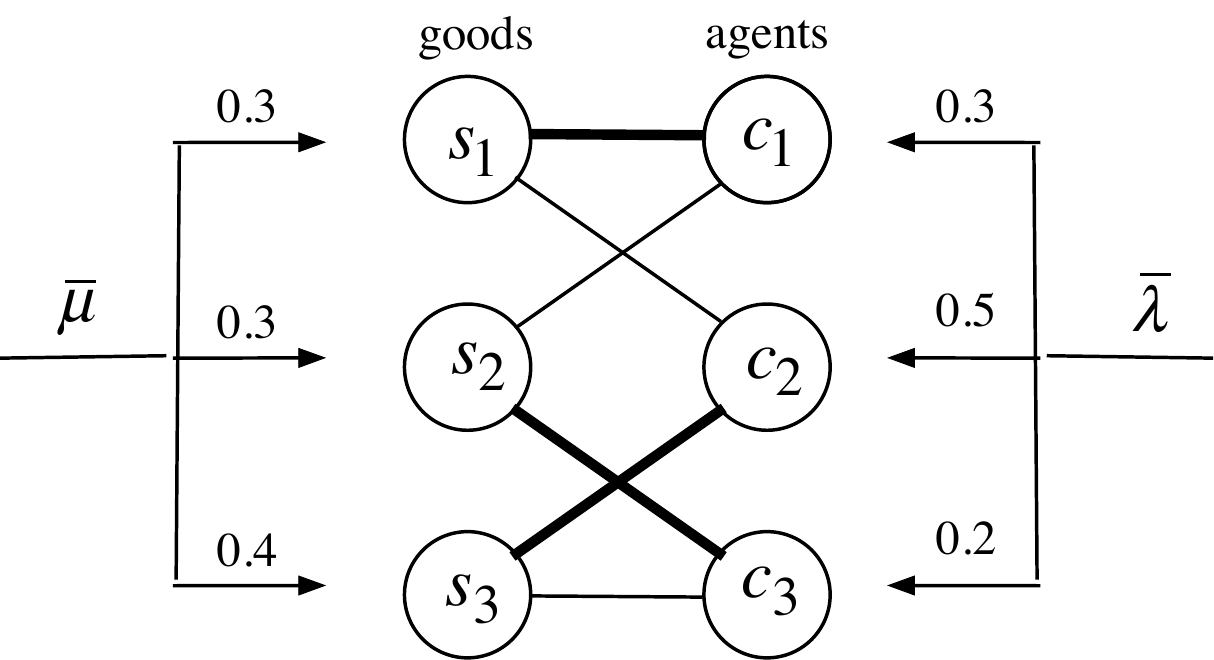} 
   \caption{A 3x3 system}
   \label{fig.3x3system}
\end{figure}
With the help of Algorithms 1 and 2 we can calculate matching rates and means and variances of delays.
For $\rmu=1$, $\rlambda=0.7$ and $\rho=\rlambda/\rmu=0.7$ we obtain:
\[
r_{s_j,c_i} = \begin{array}{l|ccc|}
& c_1 & c_2 & c_3 \\
s_1 & 0.090 & 0.139 & 0 \\
s_2 & 0.120 & 0 & 0.067 \\
s_3 & 0 & 0.211 & 0.073 
\end{array},   \qquad
r_{s_j,\emptyset} = \begin{array}{l|c|}
& \mbox{lost} \\
s_1 & 0.071  \\
s_2 & 0.113  \\
s_3 & 0.116 
\end{array},
\] 
\[
E(L_{s_j,c_i}) = \begin{array}{l|ccc|}
& c_1 & c_2 & c_3 \\
s_1 & 7.63 & 7.64 & 0 \\
s_2 & 7.14 & 0 & 6.38 \\
s_3 & 0 & 7.40 & 6.45 
\end{array},   \qquad
\sigma(L_{s_j,c_i}) = \begin{array}{l|ccc|}
& c_1 & c_2 & c_3 \\
s_1 & 6.14 & 6.30 & 0 \\
s_2 & 5.97 & 0 & 5.41 \\
s_3 & 0 & 6.21 & 5.46 
\end{array}.
\] 
From this calculation we see that the loss probabilities for the various goods may be quite different.  
We also obtain for the delays of the various type of agents:
\[
\begin{array}{lll}
E(L_{c_1}) = 7.35, & \quad E(L_{c_2}) =  7.50,  &\quad E(L_{c_3}) =  6.38,  \\
\sigma((L_{c_1})= 6.05, & \quad \sigma((L_{c_2})= 6.25, & \quad \sigma((L_{c_3})= 5.44.
\end{array}
\]
Assume now that  arrivals are Poisson, with rates $\rlambda,\rmu$, then using Corollary \ref{thm.waits} we have:
\[
\begin{array}{lll}
E(W_{c_1}) = 4.33, & \quad E(W_{c_2}) =  4.41,  &\quad E(W_{c_3}) =  3.75. \\
\sigma((W_{c_1})= 3.90, & \quad \sigma((W_{c_2})= 4.01, & \quad \sigma((W_{c_3})= 3.53.
\end{array}
\]
We can see from these results that the FCFS matching policy achieve balanced supply of goods, the expected delays for the three types of agents are very close to each other, and even the expected delays for individual $(s_j,c_i)\in \G$ are quite close.   We also see that the standard deviation of delays is close to the expectation, which indicates that the distribution of delays and waiting times is close to exponential.

By Proposition \ref{thm.lighttraffic} and Theorem \ref{thm.heavytraffic} we know  the limits of the matching rates as $\rho=\rlambda/\rmu$ approaches 0 or 1.   Fig \ref{fig.RatesPlot} shows how the matching rates for $(s_j,c_i)\in\G$ vary as $0<\rho<1$.
\begin{figure}[htbp]
   \centering
   \includegraphics[scale=0.8]{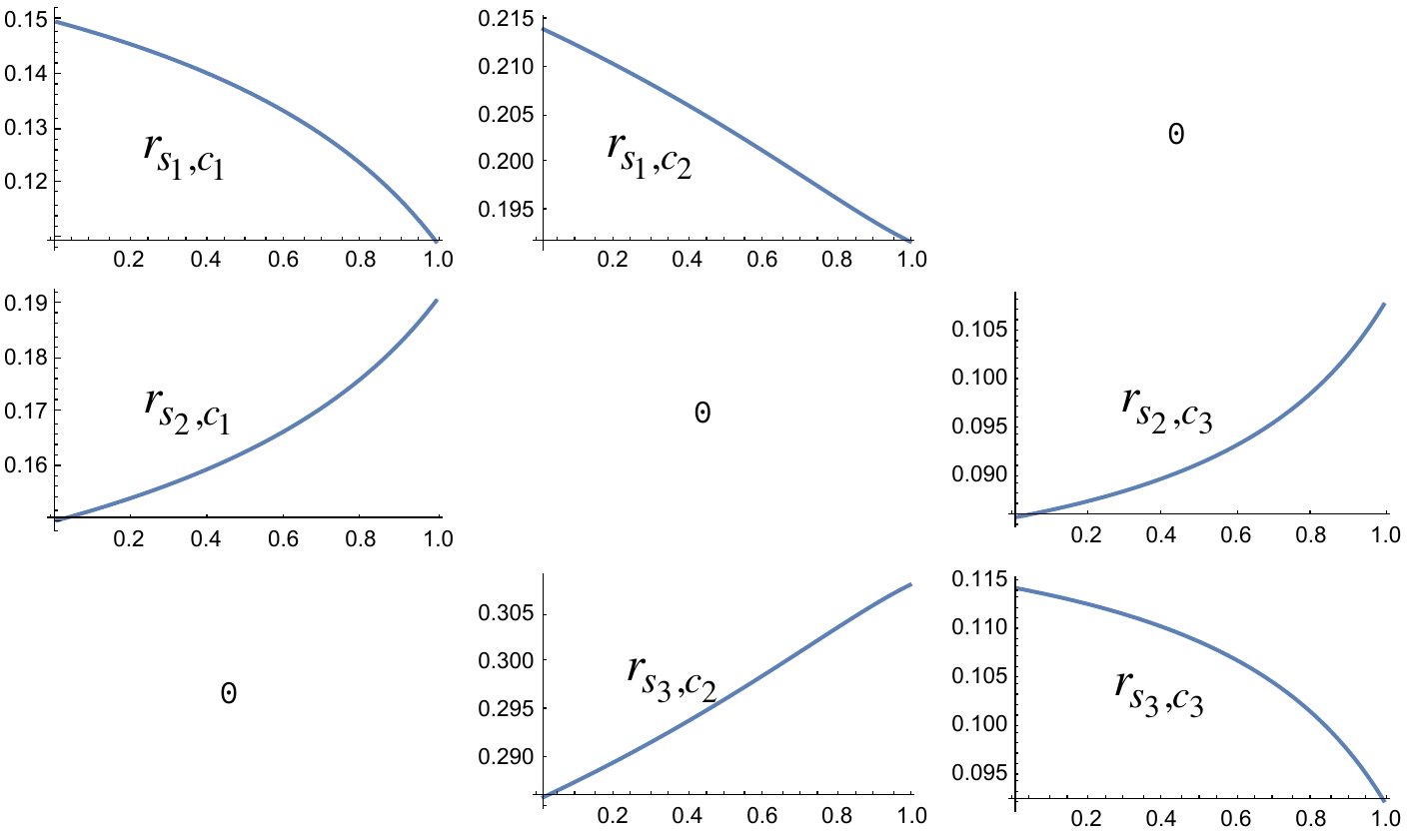} 
   \caption{Plots of the matching rates as a function of traffic intensity}
   \label{fig.RatesPlot}
\end{figure}
We see  that the variations in values of match rates are not very large.  In very light traffic each agent receives goods immediately, and is not influenced by the agents of the  other types, so the
matching rates are as if each type of agent has its own set of goods without sharing.  In heavy traffic there is CRP and sharing of all the resources of the goods.  The change over the range of values of $\rho$ is monotone.   For a possible explanation of how the rates change consider agent type $c_2$.  In light traffic it receives $\frac{3}{7}$ from $s_1$ and $\frac{4}{7}$ from $s_3$.  In heavy traffic, goods of types $s_1$ are shared with $c_1$, and goods of types $s_3$ are shared with $c_3$, but the requirements of $c_1$ are $0.3$ and those of $c_3$ are only $0.2$, so more goods of type $s_3$ then of type $s_1$ are diverted to $c_2$, and so $r_{s_1,c_2}$ decreases, and $r_{s_3,c_2}$ increases with $\rho$. 

It is perhaps worthwhile to compare the bipartite matching of this example to what would be the results of eliminating the pooling effect.  Assume each type of agent is matched with only one type of good, as follows:  $\G_1=\{(s_1,c_1),(s_2,c_3),(s_3,c_2)\}$ (these edges are marked by heavier lines in Fig \ref{fig.3x3system}).  In Fig \ref{fig.WaitPlot}  we plot the average waiting times for each type of agent, for our system under the pooled bipartite matching, compared with the waiting times for discriminating matching according to $\G_1$.  Under discriminating matching each agent-good pair behaves like an M/M/1 queue with expected waiting time $\frac{1}{\mu_{s_j}-\lambda_{c_i}}$.  We see that the pair $(s_3,c_2)$ becomes unstable for $\rho > 0.8$, the pair $(s_2,c_3)$ has discriminating traffic intensity $<\frac{2}{3}$ but pooled waiting time is still shorter than discriminating for all $\rho<0.8$.  For the pair  $(s_1,c_1)$, where  discriminating traffic intensity varies form 0 to 1, the ratio of waiting times starts at $\approx 2$ in light traffic and converges to $\approx 3.3$ as $\rho \nearrow 1$.   This is   slightly better than the ratio we expect from pooling 3 i.i.d. M/M/1 queues. 
\begin{figure}[htbp]
   \centering
   \includegraphics[scale=0.35]{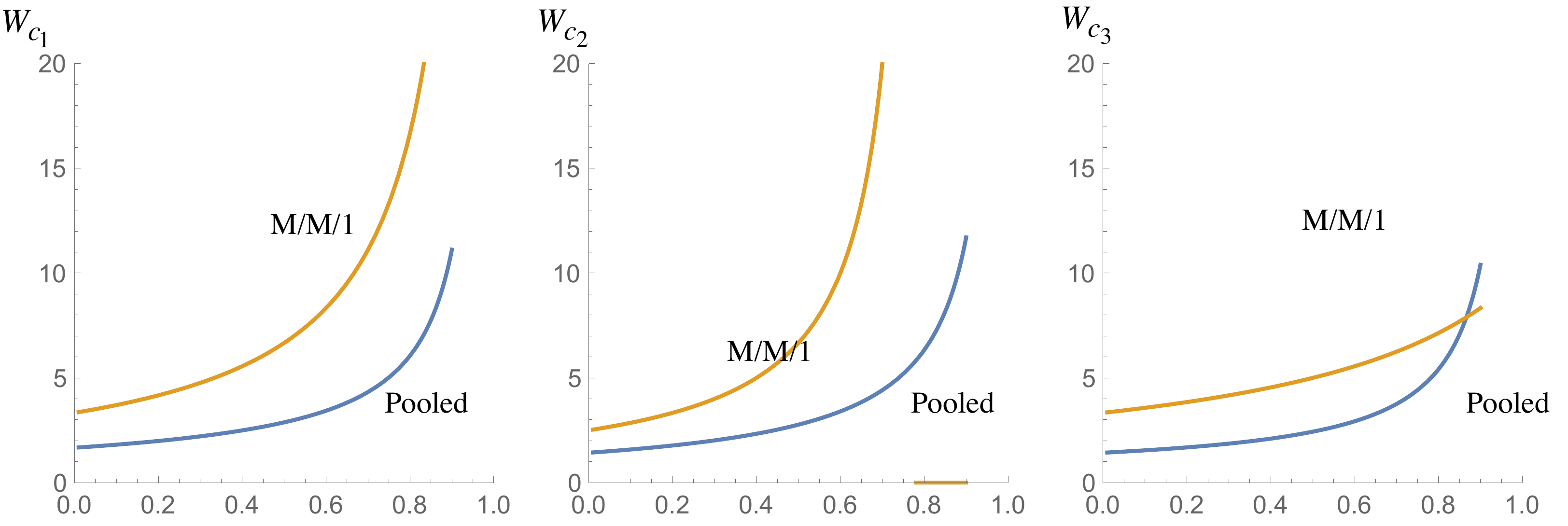} 
   \caption{Plots of the Waiting Time, for Each Agent Type, Pooled and Not-Pooled}
   \label{fig.WaitPlot}
\end{figure}

\bibliographystyle{plain}
\bibliography{DirectedFCFSBibliography}   

\end{document}